\documentclass[reqno]{amsart}

\usepackage{amsmath}
\usepackage{amssymb}
\usepackage{verbatim}
\usepackage{dsfont}
\usepackage{enumerate}
\usepackage[all]{xy}
\usepackage[mathscr]{eucal}

\usepackage[usenames]{color}

\usepackage{amsthm}
\usepackage{stmaryrd}
\usepackage{amscd}
\usepackage{amsfonts}
\usepackage{latexsym}

\usepackage{amscd}

\usepackage{graphicx}

\usepackage{setspace}

\usepackage{wasysym}

\usepackage[usenames,dvipsnames]{xcolor}

\usepackage{todonotes}

\makeatletter

\def\bd#1{\text{\boldmath${#1}$}}
\newcommand{\jonnytodo}{\todo[inline,color=blue!20]}

\newtheorem{theorem}{Theorem}[section]
\newtheorem{lemma}[theorem]{Lemma}

\newtheorem{corollary}[theorem]{Corollary} 
\theoremstyle{definition}

\newtheorem{remark}[theorem]{Remark}

\@addtoreset{equation}{section}

\newcommand{\End}{\operatorname{End}}
 
\newcommand{\Hom}{\operatorname{Hom}} 
 
\newcommand{\Rep}{\operatorname{Rep}}

\newcommand{\g}{\mathfrak{g}}
\newcommand{\Z}{\mathbb{Z}}

\newcommand{\C}{\mathbb{C}}
\newcommand{\K}{\mathbb{K}}

\newcommand{\cat}{\mathcal}
\newcommand{\ep}{\varepsilon}

\newcommand{\up}{\uparrow}
\newcommand{\down}{\downarrow}
\newcommand{\ob}[1]{\mathsf{#1}}
\newcommand{\wrd}{\langle\up,\down\rangle}
\newcommand{\Rib}{\mathcal{RIB}}
\newcommand{\gr}{\operatorname{gr}}
\newcommand{\grAOB}{\operatorname{gr}\mathcal{AOB}}
\newcommand{\grOB}{\operatorname{gr}\mathcal{OB}}
\newcommand{\unit}{\mathds{1}}

\newcommand{\col}{\operatorname{col}}

\newcommand{\AOB}{\mathcal{AOB}}
\newcommand{\GOB}{\mathcal{GOB}}
\newcommand{\COB}{\mathcal{OB}}
\newcommand{\TOB}{\mathcal{GOB}}
\newcommand{\OB}{\mathcal{OB}}

\renewcommand{\k}{\Bbbk}

\allowdisplaybreaks

\begin{document}

\title[Oriented Brauer categories]{A basis theorem for the affine oriented Brauer category and its cyclotomic quotients}

\author[Brundan]{Jonathan Brundan}
\address{J.B.: Department of Mathematics,
University of Oregon, Eugene, OR 97403, USA}
\email{brundan@uoregon.edu}

\author[Comes]{Jonathan Comes}
\address{J.C.: Department of Mathematics,
University of Oregon, Eugene, OR 97403, USA}
\email{jcomes@uoregon.edu}

\author[Nash]{David Nash}
\address{D.N.: Department of Mathematics and Computer Science, Le Moyne College, Syracuse, NY 13214, USA}
\email{nashd@lemoyne.edu}

\author[Reynolds]{Andrew Reynolds}
\address{A.R.: Department of Mathematics,
University of Oregon, Eugene, OR 97403, USA}
\email{asr@uoregon.edu}

\thanks{First author
supported in part by NSF grant DMS-1161094.}

\begin{abstract}
The affine oriented Brauer category
is a monoidal category obtained from the 
oriented Brauer category ($=$ the free symmetric monoidal category generated by a single object and its dual) by adjoining a polynomial generator subject to appropriate relations. 
In this article, we prove a basis theorem for the morphism spaces in this category, as well as for all of its cyclotomic quotients.
\end{abstract}

\maketitle


\section{Introduction}

Throughout the article, $\k$ denotes a fixed ground ring which we assume is an integral domain, and
all categories and functors 
will be assumed to be $\k$-linear.
We are going to define and study various categories
$\OB$, $\AOB$, and $\COB^f$, which we call the {\em oriented Brauer category},
the {\em affine oriented Brauer category}, and the {\em cyclotomic oriented Brauer category} associated to a monic polynomial $f(u) \in \k[u]$ of degree $\ell$.
The first of these, $\OB$, is
the free 
symmetric monoidal category generated by a single object $\uparrow$ and its dual 
$\downarrow$. Then $\AOB$ is the (no longer symmetric) monoidal category obtained from $\OB$ by adjoining an endomorphism $x:\up \rightarrow \up$ subject to relations similar to those satisfied by the polynomial generators of the degenerate affine Hecke algebra.
Finally $\COB^f$ is the (no longer monoidal) category obtained from $\AOB$ by factoring out the right tensor ideal generated by $f(x)$.

In our setup, the categories $\OB, \AOB$, and $\COB^f$ 
come equipped with some algebraically independent parameters:
one parameter $\Delta$ for $\OB$, 
infinitely many parameters $\Delta_1,\Delta_2,\dots$ for $\AOB$, 
and $\ell$ parameters $\Delta_1,\dots,\Delta_\ell$ for $\COB^f$.
On evaluating these parameters at 
scalars in the ground ring $\k$,
we obtain also various specializations $\OB(\delta)$,
$\AOB(\delta_1,\delta_2,\dots)$, and $\COB^f(\delta_1,\dots,\delta_\ell)$.
In particular, $\OB(\delta)$ is the symmetric monoidal
category denoted
$\underline{\operatorname{Re}}\!\operatorname{p}_0(GL_\delta)$ in
\cite[$\S$3.2]{CW},
which is the ``skeleton'' of Deligne's category 
$\underline{\operatorname{Re}}\!\operatorname{p}(GL_\delta)$.
The other two specialized categories $\AOB(\delta_1,\delta_2,\dots)$ and
$\COB^f(\delta_1,\dots,\delta_\ell)$ are not monoidal, but they are both right
module categories over $\AOB$.

The endomorphism algebras of objects in our various specialized categories have already appeared elsewhere in the literature.
To start with,
writing $\up^r \down^s$ for the tensor product of $r$ copies of $\up$ and $s$ copies of $\down$, 
\begin{equation}
B_{r,s}(\delta) := \End_{\OB(\delta)}(\up^r \down^s)
\end{equation}
is the well-known {\em walled Brauer algebra}
which was introduced independently by Turaev \cite{Tur89} and Koike [Ko] in the late 1980s; see e.g. \cite{BS}.
By analogy with this, we define 
the {\em affine} and {\em cyclotomic walled Brauer algebras} to be the endomorphism algebras
\begin{align}\label{AB}
AB_{r,s}(\delta_1,\delta_2,\dots) &:= \End_{\AOB(\delta_1,\delta_2,\dots)}(\up^r \down^s),\\
B^f_{r,s}(\delta_1,\dots,\delta_\ell) &:=\End_{\COB^f(\delta_1,\dots,\delta_\ell)}(\up^r \down^s).\label{Brs}
\end{align}
In the last subsection of the article, we will explain how our affine walled Brauer algebra is isomorphic to the algebra with the same name defined by generators and (twenty-six!) relations by Rui and Su \cite{RS}; see also \cite{Sartori}.
Similarly our cyclotomic walled Brauer algebras
are isomorphic to the ones introduced in \cite{RS2}. 

The main goal of the article is to prove various diagrammatic basis theorems for the morphism spaces in the categories $\AOB$ and $\COB^f$.
Rui and Su also prove basis theorems for their algebras in \cite{RS, RS2}.
The tricky step in the proofs of all of
these results is to establish the linear independence.
It turns out that the linear independence in \cite{RS, RS2} can be deduced quite easily from the basis theorems proved in the present paper. On the other hand, it does not seem to be easy to deduce our results from those of \cite{RS,RS2}. In fact we found it necessary to adopt a completely different approach.

In subsequent work \cite{BR}, the first and last authors will
investigate the representation theory of the cyclotomic quotients
$\COB^f(\delta_1,\dots,\delta_\ell)$,
showing for suitably chosen parameters 
with $\k = \C$ that they give rise to tensor product categorifications of 
integrable lowest and highest weight representations of level $\ell$ for the Lie algebra $\mathfrak{sl}_\infty$.
We expect that these categories
are closely related to the categorifications of tensor products of lowest and highest weight representations introduced by Webster in \cite{Webster}.

In the remainder of the introduction, we are going to explain in detail the definitions of all of these categories, then formulate our main results precisely.

\subsection*{\boldmath The category $\OB$}
Let $\wrd$ denote the set of all words in the
alphabet $\{\uparrow, \downarrow\}$, including the empty word $\varnothing$.
Given two words $\ob a = \ob a_1\cdots \ob a_k, \ob b = \ob b_1 \cdots \ob b_l
\in\wrd$, an {\em oriented Brauer diagram} of type $\ob a \rightarrow \ob b$ is a
diagrammatic representation of a
 bijection 
$$
\{i\:|\:\ob a_i = \up\} \sqcup \{i'\:|\:\ob b_i = \down\}
\stackrel{\sim}{\rightarrow} 
\{i\:|\:\ob b_i = \up\}\sqcup\{i'\:|\:\ob a_i = \down\}
$$
obtained by placing the word $\ob a$ below
the word $\ob b$, then drawing strands connecting pairs of letters as prescribed by the given bijection. 
The arrows that are the 
letters of $\ob a$ and $\ob b$ then give a consistent orientation to each strand in the diagram.
For example, 
$$\includegraphics{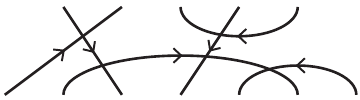}$$ is an oriented Brauer diagram of type 
$\up\up\down\down\down\down\up\to\down\up\up\down\down$. 
We say that two oriented Brauer diagrams are {\em equivalent}
if they are of the same type and represent the same bijection.
In diagrammatic terms, this means that one diagram can be obtained from the other by continuously deforming its strands, possibly moving them through
other strands and crossings, but keeping
endpoints fixed.

Given oriented Brauer diagrams of types $\ob b \rightarrow \ob c$ and
$\ob a \rightarrow \ob b$, we can stack the first on top of the second to 
obtain an oriented Brauer diagram of type $\ob a\to\ob c$ along with finitely many loops made up of strands which were connected only to letters in $\ob b$, which we call \emph{bubbles}. For example, if we stack the two diagrams 
$$
\includegraphics{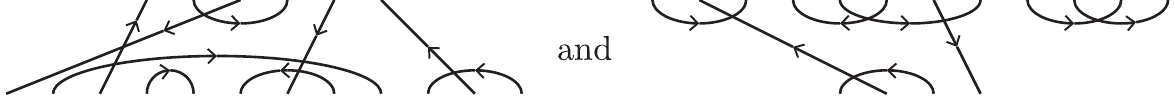}
$$ 
we obtain the following {\em oriented Brauer diagram with bubbles}:
$$\includegraphics{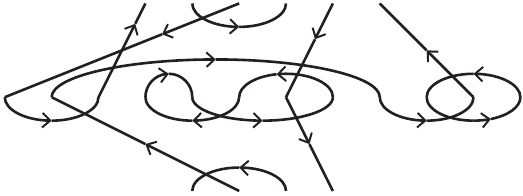}$$
Two oriented Brauer diagrams with bubbles are {\em equivalent}
if they have the same number of bubbles
(regardless of orientation),
and the underlying oriented Brauer diagrams obtained by ignoring the bubbles are equivalent in the earlier sense; again this can be viewed in terms of continuously deforming strands through other strands and crossings.
For example, the oriented Brauer diagram with bubbles pictured above
is equivalent to the following one:  
$$\includegraphics{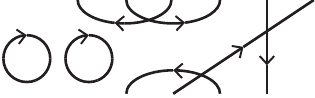}$$

Now we can define the {\em oriented Brauer category} $\OB$ to be the
category with objects
$\wrd$ and
morphisms $\Hom_{\OB}(\ob a, \ob b)$ consisting of all formal
$\k$-linear combinations of equivalence classes of
oriented Brauer diagrams with bubbles of
type $\ob
a \rightarrow \ob b$.
The composition $g \circ h$ 
of diagrams is by vertically stacking $g$ on top of $h$ as illustrated 
above; it is easy to see that this is associative.
There is also a well-defined tensor product making $\OB$ into a (strict) 
monoidal category.
This is 
defined on diagrams so that $g \otimes h$ is obtained by horizontally stacking $g$
to the left of $h$; often we will denote $g \otimes h$ simply by $gh$.
There is an obvious braiding $\sigma$ making $\OB$ into a symmetric monoidal category,
which is defined
by setting $$\includegraphics{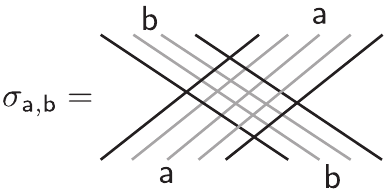}$$ for each $\ob{a,b}\in\wrd$. 
Finally $\OB$ is rigid, with the dual $\ob a^*$ of $\ob a$
being the word obtained by rotating $\ob a$ through $180^\circ$:
$$\includegraphics{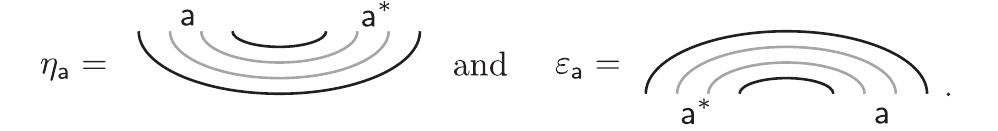}$$
We refer to Section 2 for a brief review of these basic notions.

The monoidal category $\OB$ can also be defined by generators and
relations.
To explain this, 
let $c:\varnothing\to\up\down$ (``create''), $d:\down\up\to\varnothing$
(``destroy"), and $s:\up\up\to\up\up$ denote the following oriented Brauer diagrams: 
\begin{equation*}\label{c d t}\includegraphics{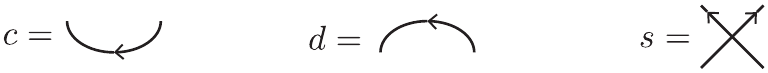}\end{equation*}
It is easy to check that these satisfy the following fundamental relations:
\begin{align}\label{A down}
(\up d)\circ(c\up)&=\up,\\
\label{A up}
(d\down)\circ(\down c)&=\down,\\
\label{S}
s^2&=\up\up,\\
\label{B}
(\up s)\circ(s\up)\circ(\up s)&=(s\up)\circ(\up s)\circ (s\up),\\
\label{I}
(d\up\down)\circ(\down s\down)\circ(\down\up c)&\text{ is
  invertible.}
 \end{align}
The last of these relations is really the assertion that there is
another distinguished generator $t:\up \down \rightarrow \down \up$ that is a
two-sided inverse to $(d\up\down)\circ(\down s\down)\circ(\down\up c)$, as illustrated below:
\begin{equation*}\includegraphics{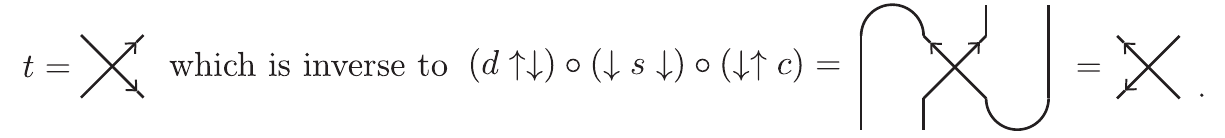}
\end{equation*}
The following theorem is an almost immediate consequence of a general
result of Turaev from \cite{Tur} which gives a presentation for the category of ribbon
tangles.

\begin{theorem}\label{thm1}
As a $\k$-linear  monoidal category,
  $\OB$ is generated by the objects $\up,\down$ and morphisms $c,
  d, s$
  subject only to the  relations (\ref{A down})--(\ref{I}).
\end{theorem}

The endomorphism 
algebra $\End_{\OB}(\varnothing)$ is the polynomial algebra $\k[\Delta]$
generated by a single bubble $\Delta$.
Moreover we have that $\Delta \otimes g = g \otimes \Delta$ for each morphism $g$ in $\OB$.
Hence we can view $\OB$ as a $\k[\Delta]$-linear monoidal category so that
$\Delta$ acts on morphism $g$ 
by $\Delta g := \Delta
\otimes g$.
Then, given a scalar $\delta \in \k$, we let $\OB(\delta)$ denote the
symmetric monoidal category obtained from $\OB$ by specializing $\Delta$
at $\delta$, i.e. $\OB(\delta) := \k \otimes_{\k[\Delta]} \OB$ viewing
$\k$ as a $\k[\Delta]$-module so that $\Delta$ acts as multiplication
by $\delta$.
Equivalently, in terms of generators and relations, 
$\OB(\delta)$ 
is the $\k$-linear monoidal category 
obtained from $\OB$ by imposing the additional relation
\begin{equation}
d \circ t \circ c = \delta,
\end{equation}
i.e. we require that the 
object $\uparrow$ has dimension $\delta$.
Each morphism space $\Hom_{\OB}(\ob a, \ob b)$ in $\OB$ is free as a
$\k[\Delta]$-module with basis given by the equivalence classes of oriented Brauer diagrams
$\ob a \rightarrow \ob b$ (now with no bubbles).
Hence the category $\OB(\delta)$ has objects
$\wrd$ and its
morphisms $\Hom_{\OB(\delta)}(\ob a, \ob b)$ are formal $\k$-linear
combinations of equivalence classes of oriented
Brauer diagrams of type $\ob a \rightarrow \ob b$. The 
composition $g \circ h$ of two diagrams in $\OB(\delta)$ is defined first by vertically
stacking $g$ on top of $h$, then removing all bubbles 
and multiplying by the scalar $\delta^n$, where $n$ is the number of bubbles
removed.

\subsection*{\boldmath The category $\AOB$}
The {\em affine oriented Brauer category} $\AOB$ is the monoidal category
generated by objects $\up, \down$ and morphisms $c,d,s, x$
subject to the relations (\ref{A down})--(\ref{I}) plus one extra relation
\begin{equation}\label{AX}
(\up x)\circ s=s\circ(x\up)+\up\up.
\end{equation}
Before we discuss the diagrammatic nature of this category, let us explain 
why we became interested in it in the first place. 
Let $\mathfrak{g}$ be the general linear Lie algebra $\mathfrak{gl}_n(\k)$ with natural module $V$.
Let $\g\operatorname{-mod}$ be the category of all $\g$-modules
and
$\End(\g\operatorname{-mod})$ be the monoidal category of endofunctors of
$\g\operatorname{-mod}$, so
for functors $F,G, F', G'$ and
natural transformations $\eta:F \rightarrow F'$, $\xi:G \rightarrow G'$
 we have that
$F \otimes G := F \circ G$ and $\eta \otimes \xi := \eta  \xi:F \circ G \rightarrow F' \circ G'$.
Finally let $\End(\g\operatorname{-mod})^{\operatorname{rev}}$
denote the same category but viewed as a monoidal category with the opposite
tensor product.
Then the point is that there is a monoidal functor 
\begin{equation}\label{F}
R:\AOB \rightarrow
\End(\g\operatorname{-mod})^{\operatorname{rev}}
\end{equation}
sending the objects $\up$ and $\down$
to the endofunctors $- \otimes V$ and $- \otimes V^*$, respectively,
and defined on
the generating morphisms by
\begin{align*}
R(c)&:\operatorname{Id} \rightarrow - \otimes V \otimes V^*, \quad&u &\mapsto u \otimes \omega,\\
R(d)&:- \otimes V^* \otimes V \rightarrow \operatorname{Id},
\quad&
u \otimes f \otimes v &\mapsto f(v) u,\\
R(s)&:- \otimes V \otimes V \rightarrow - \otimes V \otimes V,\quad&
u \otimes v \otimes w &\mapsto u \otimes w \otimes v,\\
R(x)&:- \otimes V \rightarrow - \otimes V,
\quad&
u \otimes v &\mapsto \Omega(u \otimes v),
\end{align*}
 where $\omega := \sum_{i=1}^n v_i \otimes f_i$ 
assuming $\{v_i\}$ and $\{f_i\}$ are dual bases for $V$ and $V^*$,
and $\Omega$ is the Casimir tensor $\sum_{i,j=1}^n e_{i,j} \otimes e_{j,i} \in \g \otimes \g$.
(One can also define an analogous functor $R$ with
$\g$ replaced by the general linear Lie superalgebra.)

Returning to the main discussion,
we note by Theorem~\ref{thm1} that there is a functor $\OB\to \AOB$ sending the 
generators of $\OB$ to the generators of $\AOB$ with the same name.
Hence we can
interpret any oriented Brauer diagram with bubbles also as a morphism
in $\AOB$.  
We also want to add dots to our diagrams, corresponding to the new
generator $x$ which we represent by the diagram
$$
\includegraphics{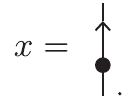}
$$
To formalize this, we define a {\em dotted oriented Brauer diagram
  with bubbles}
to be an
oriented Brauer diagram with bubbles, such that each segment
is decorated in addition with some non-negative
number of dots, where
a {\em segment} means a connected component of the diagram
obtained when all crossings are deleted.
Two dotted oriented Brauer diagrams with bubbles are {\em equivalent}
if one can be obtained from the other by continuously deforming strands
through other strands and crossings as above,
and also by sliding dots along
strands, all subject to the requirement that {\em dots are never allowed to pass through crossings}.
For example, here are two such diagrams which are not
equivalent:
$$
\includegraphics{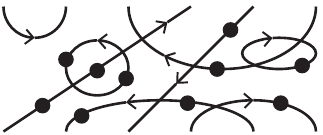}
\qquad\qquad
\includegraphics{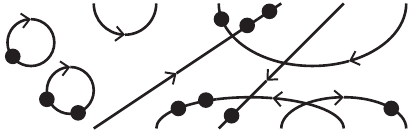}
$$
Any dotted oriented Brauer
diagram with bubbles is equivalent to one that is a vertical composition
of elementary diagrams of the form $\ob a \,c\, \ob b$, $\ob a\, d\, \ob b$, 
$\ob a\, s\, \ob b$, $\ob a\, t\, \ob b$, $\ob a\, x\, \ob b$ for various $\ob a, \ob b \in \wrd$. Hence it can be
interpreted as a morphism in $\AOB$. Moreover,
the resulting morphism is well defined independent of the choices
made, and it depends only on the equivalence class of the
 original diagram. 
For example, the following diagram $x'$
represents the morphism $(d \down) \circ (\down x \down) \circ (\down c)\in
\End_{\AOB}(\down)$:
\begin{equation*}
\includegraphics{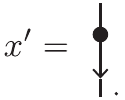}
\end{equation*}
Also the relation (\ref{AX}) transforms into the first of the
following two diagrammatic relations;
the second follows from the first by
composing with $s$ on the top and bottom:
$$
\includegraphics{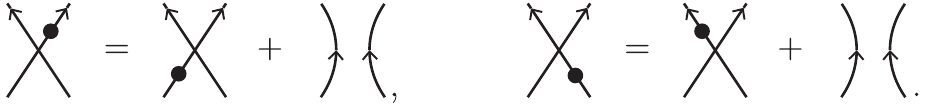}
$$ 
These local relations explain how to 
move dots past crossings in any diagram.

A dotted oriented Brauer diagram with bubbles is {\em normally ordered}
if:
\begin{itemize}
\item[$\star$]
all of its bubbles are clockwise, crossing-free, and there are no other strands shielding any of them from the leftmost edge of the picture;
\item[$\star$]
all of its dots are either on bubbles
or on outward-pointing boundary segments, i.e. segments which
intersect the boundary at a point that is directed out of the picture.
\end{itemize}
For example, of the two dotted oriented Brauer diagrams with bubbles
displayed in the previous paragraph, only the second one is normally ordered.

\begin{theorem}\label{thm2}
For $\ob a, \ob b \in \wrd$, the space $\Hom_{\AOB}(\ob a, \ob b)$ is a
free $\k$-module with basis given by 
equivalence classes of normally ordered dotted
oriented Brauer diagrams with bubbles of 
type $\ob a \rightarrow \ob b$.
\end{theorem}

By Theorem~\ref{thm2}, the endomorphism algebra $\End_{\AOB}(\varnothing)$
is the polynomial algebra $\k[\Delta_1,\Delta_2,\dots]$
generated by the clockwise dotted bubbles
$$
\includegraphics{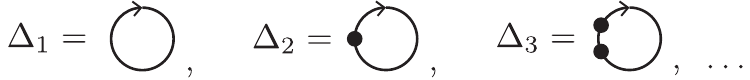}
$$
In the same way as we explained earlier for $\OB$, one can then view $\AOB$
as a $\k[\Delta_1,\Delta_2,\dots]$-linear category with 
$\Delta_i g := \Delta_i \otimes g$;
note though that $\Delta_i \otimes g$ and $g \otimes \Delta_i$
are in general different for $i \geq 2$
so that the monoidal structure is not
$\k[\Delta_1,\Delta_2,\dots]$-linear.
Then given scalars $\delta_1,\delta_2,\dots \in \k$, we 
let $\AOB(\delta_1,\delta_2,\dots)$ be the $\k$-linear category obtained from
$\AOB$ by
 specializing
each $\Delta_i$ at $\delta_i$, i.e.
$\AOB(\delta_1,\delta_2,\dots) = \k \otimes_{\k[\Delta_1,\Delta_2,\dots]} \AOB$
viewing $\k$ as a $\k[\Delta_1,\Delta_2,\dots]$-module so each $\Delta_i$ acts as $\delta_i$.
Theorem~\ref{thm2} implies that $\Hom_{\AOB(\delta_1,\delta_2,\dots)}(\ob a, \ob b)$
is a free $\k$-module with basis given by
the equivalence classes of
normally ordered dotted Brauer diagrams of type $\ob a \rightarrow \ob
b$ (now with no bubbles).


\begin{remark}\label{rem1}\rm
One can also describe $\End_{\AOB}(\varnothing)$ as the algebra
$\k[\Delta_1',\Delta_2',\dots]$ freely generated by the counterclockwise dotted bubbles
$$
\includegraphics{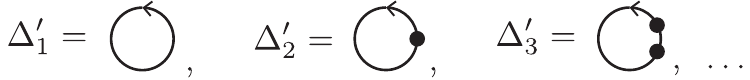}
$$
The relationship between 
$\Delta_i$ and $\Delta_j'$
is explained by the identity 
\begin{equation}\label{rel}
\left(1+\sum_{i \geq 1} \Delta_i u^{-i}\right) \left(1-\sum_{j \geq 1} \Delta_j' u^{-j}\right) = 1,
\end{equation}
which (up to some signs) is the same as the relationship between elementary and complete symmetric functions in the ring of symmetric functions.
For example, the coefficient of $u^{-2}$ in (\ref{rel})
is equivalent to the assertion that $\Delta_2 = \Delta_2'+\Delta_1 \Delta_1'$, which follows from the following calculation with relations:
$$\includegraphics{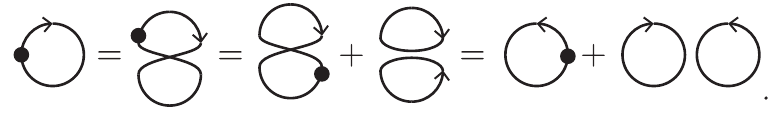}$$
The identities corresponding to the other coefficients of $u$ 
in (\ref{rel}) follow in a similar way.
\end{remark}

\begin{remark}\label{rem2}\rm
It is instructive to compute the images of $\Delta_1,\Delta_2,\dots$
under the functor $R:\AOB \rightarrow \End(\g\operatorname{-mod})^{\operatorname{rev}}$ from (\ref{F}). 
Since this functor maps $\End_{\AOB}(\varnothing)$ to 
$\End(\operatorname{Id})$, which is canonically identified with the center $Z(\g)$ of the universal enveloping algebra $U(\g)$,
these are naturally elements of $Z(\g)$:
$$
R(\Delta_k) = \sum_{\substack{1 \leq i_1,\dots, i_k \leq n \\ i_1 = i_k}}
e_{i_{k-1},i_k}\cdots e_{i_2,i_3} e_{i_1,i_2}.
$$
(The right hand side of this formula means the scalar $n$ in case $k=1$.)
Using this and taking a limit as $n \rightarrow \infty$,
one can give an alternative proof of the algebraic independence of 
$\Delta_1,\Delta_2,\dots$ in $\End_{\AOB}(\varnothing)$.
\end{remark}

\subsection*{\boldmath The category $\COB^f$}
Let $\ell \geq 1$ be a fixed {\em level}
and $f(u)  \in \k[u]$
be a monic polynomial of degree $\ell$.
The {\em cyclotomic oriented Brauer category} $\COB^f$ is the quotient of $\AOB$ by the right tensor ideal generated by $f(x) \in \End_{\AOB}(\uparrow)$.
Note 
this category is not monoidal in any natural way
(except if $\ell = 1$ when $\COB^f$ turns out to be isomorphic to $\OB$),
although it is still a right module category over $\AOB$.
It has the same objects as $\AOB$ while its morphism spaces are quotients of the ones in $\AOB$. Hence 
any morphism in $\AOB$ can also be viewed as a morphism in $\COB^f$.
In particular we can interpret equivalence classes of 
dotted oriented Brauer diagrams with bubbles
also as morphisms in $\COB^f$.

\begin{theorem}\label{thm3}
For $\ob a, \ob b \in \wrd$, the space $\Hom_{\COB^f}(\ob a, \ob b)$ is a
free $\k$-module with basis given by equivalence classes of 
normally ordered dotted
oriented Brauer diagrams with bubbles of 
type $\ob a \rightarrow \ob b$, subject to the additional constraint that
each strand is decorated by at most $(\ell-1)$ dots.
\end{theorem}

This means in $\COB^f$ that any clockwise dotted bubble with
$\ell$ or more dots can be expressed in terms of $\Delta_1,\dots,\Delta_\ell$.
So there are really only $\ell$ algebraically independent parameters
that can be specialized here: for $\delta_1,\dots,\delta_\ell \in \k$ we let
$\COB^f(\delta_1,\dots,\delta_\ell)$ be the $\k$-linear category obtained from
$\COB^f$ by specializing $\Delta_i$ at $\delta_i$ for each $i=1,\dots,\ell$.
In other words, we impose the additional relations that $\Delta_i \ob
a = \delta_i \ob a$ for each $\ob a \in \wrd$ and $i=1,\dots,\ell$.
Theorem~\ref{thm3} then implies that $\Hom_{\COB^f(\delta_1,\dots,\delta_\ell)}(\ob a, \ob b)$
is a free $\k$-module with basis arising from the 
normally ordered dotted oriented Brauer diagrams with no bubbles each of whose strands are decorated by at most $(\ell-1)$ dots.

\begin{remark}\label{rem3}\rm
There is another way to specify the parameters 
of $\COB^f(\delta_1,\dots,\delta_\ell)$ which is more symmetric.
Suppose for this that we are given a pair of monic polynomials
$f(u), f'(u) \in \k[u]$, both of the same degree $\ell$.
From these we extract scalars $\delta_i \in \k$ by setting
\begin{equation}\label{deltas}
1+\sum_{i \geq 1} \delta_i u^{-i} := f'(u) / f(u)
\in\k[[u^{-1}]].
\end{equation}
Then we introduce the alternative notation $\COB^{f,f'}$
to denote $\COB^f(\delta_1,\dots,\delta_\ell)$.
By definition, this is the quotient of $\AOB$ by the right tensor ideal 
generated by $f(x)$ and $\Delta_i - \delta_i $
for each $i=1,\dots,\ell$.
Letting $\delta_j' \in \k$ be defined from the identity
\begin{equation}\label{delta id}
\left(1+\sum_{i \geq 1} \delta_i u^{-i}\right) \left(1-\sum_{j \geq 1} \delta_j' u^{-j}\right) = 1
\end{equation}
like in (\ref{rel}),
one can check 
that this right tensor ideal is generated equivalently by
$f'(x')$ and $\Delta_j' - \delta_j'$ for each $j=1,\dots,\ell$.
Moreover it automatically contains the elements
$\Delta_i - \delta_i$ and $\Delta_j'-\delta_j'$
for all $i,j > \ell$.
\end{remark}

\section{Preliminaries}\label{prelim}

\subsection{Conventions} 
Given an object $\ob a$  in some category, we write $1_\ob a$ for its identity morphism.  When it is unlikely to result in confusion, it will sometimes be convenient to abuse notation in the standard way by writing $\ob a$ for both the object as well as its identity morphism.
We assume the reader is familiar with the definition of a monoidal
category   found, for instance, in \cite{Mac}.  A relevant example of
a monoidal category is $\k$-mod, the category of all 
$\k$-modules with
the usual tensor product.
We will use the term \emph{tensor functor} to mean a 
strong monoidal functor.  

\subsection{Braided monoidal categories}  
Recall that a 
\emph{braiding} on 
a monoidal category $\cat{M}$ is a natural isomorphism $\sigma$ from the identity functor on $\cat M\times\cat M$ to the functor given by  $(\ob a,\ob b)\mapsto(\ob b,\ob a)$ such that
$\sigma_{\ob a,\ob b\otimes\ob c}=(1_\ob b\otimes \sigma_{\ob a,\ob c})\circ(\sigma_{\ob a,\ob b}\otimes 1_\ob c)$ and  $\sigma_{\ob a\otimes \ob b, \ob c}=(\sigma_{\ob a,\ob c}\otimes 1_\ob b)\circ(1_\ob a\otimes \sigma_{\ob b,\ob c})$
 for all objects $\ob{a,b,c}$ in $\cat{M}$.  The Yang-Baxter equation holds in all braided monoidal categories; in the strict case it says that all objects $\ob{a,b,c}$ satsify
\begin{equation}\label{Yang-Baxter} (1_\ob c\otimes\sigma_{\ob a,\ob b})\circ(\sigma_{\ob a,\ob c}\otimes1_\ob b)\circ(1_\ob a\otimes\sigma_{\ob b,\ob c})=(\sigma_{\ob b,\ob c}\otimes1_\ob a)\circ(1_\ob b\otimes\sigma_{\ob a,\ob c})\circ(\sigma_{\ob a,\ob b}\otimes1_\ob c).
\end{equation}
The braiding is called \emph{symmetric} if 
$\sigma_{\ob{a,b}}^{-1}=\sigma_{\ob{b,a}}$ for all objects $\ob{a,b}$.  
A monoidal category equipped with a braiding (resp. a symmetric braiding) is called  \emph{braided} (resp. \emph{symmetric}).  For example, $\k$-mod is  symmetric with $\sigma_{U,V}:u\otimes v\mapsto v\otimes u$. 

\subsection{Ideals and quotients}\label{ideals} 
Suppose $\cat{M}$ is a monoidal category.  A \emph{right tensor ideal} $\cat{I}$ of $\cat{M}$ 
is the data of a submodule $\cat{I}(\ob a,\ob b)\subseteq\Hom_\cat{M}(\ob a, \ob b)$ for each pair of objects $\ob a,\ob b$ in $\cat{M}$, such that for all objects $\ob{a, b, c, d}$ we have $h\circ g\circ f\in\cat{I}(\ob a, \ob d)$ whenever $f\in\Hom_\cat{M}(\ob a, \ob b)$, $g\in\cat{I}(\ob b,\ob c)$, $h\in\Hom_\cat{M}(\ob c,\ob d)$, and $g\otimes 1_{\ob c}\in\cat{I}(\ob a\otimes \ob c, \ob b\otimes\ob c)$ whenever $g\in\cat{I}(\ob a,\ob b)$.  
One can similarly define left and two-sided tensor ideals. 

The \emph{quotient} $\cat{M}/\cat{I}$ 
of $\cat M$ by right tensor ideal $\cat{I}$
is the category with the same objects as $\cat{M}$ and morphisms given by $\Hom_{\cat{M}/\cat{I}}(\ob a,\ob b):=\Hom_\cat{M}(\ob a,\ob b)/\cat{I}(\ob a,\ob b)$.  
The tensor product on $\cat{M}$ induces a bifunctor  $\cat{M}/\cat{I}\times\cat{M}\to\cat{M}/\cat{I}$ which gives $\cat{M}/\cat{I}$ the structure of a right module category over $\cat{M}$ in the sense of \cite[Definition 2.5]{Greenough} (see also \cite[Definition 2.6]{Ostrik03}).  
In general, $\cat{M/I}$ does {\em not} inherit the structure of a monoidal category from $\cat{M}$.  However if $\cat{M}$ is braided
then $$1_{\ob c}\otimes g=\sigma_{\ob b,\ob c}\circ(g\otimes 1_{\ob c})\circ\sigma^{-1}_{\ob a,\ob c}\in\cat{I}(\ob c\otimes\ob a,\ob c\otimes\ob b)
$$ 
whenever $g\in\cat{I}(\ob a,\ob b)$, hence every right tensor ideal in a braided monoidal category is a two-sided tensor ideal.  It is straightforward to check that the quotient of a monoidal category by a two-sided tensor ideal inherits the structure of a monoidal category.  Moreover, such quotients of a braided (resp.~symmetric) monoidal category are again braided (resp. symmetric).  

\subsection{Duality}\label{duality} A \emph{right dual} of an object $\ob a$ in a monoidal category  
consists of an object $\ob a^*$ together  with a unit morphism 
$\eta_{\ob a}:\unit\to  \ob a\otimes\ob a^*$  and a counit morphism  $\ep_{\ob a}:\ob a^*\otimes \ob a\to\unit$ such that
$(1_{\ob a}\otimes \ep_{\ob a})\circ(\eta_{\ob a}\otimes1_{\ob a})=1_{\ob a}$ and $(\ep_{\ob a} \otimes 1_{\ob a^*})\circ(1_{\ob a^*}\otimes \eta_{\ob a})=1_{\ob a^*}$ (here we are omitting associativity and unit isomorphisms).
There is also a notion of left duals.  A monoidal category in which every object has both a left and right dual is called \emph{rigid}.  
In a symmetric monoidal category, an object has a right dual $\ob a^*$
if and only if it has a left dual ${^*}\ob a$, in which case there is a canonical isomorphism  $\ob a^* \cong {^*}\ob a$.
We will only consider right duals for the remainder of the article.  

If $\ob{a,b,c}$ are objects in a monoidal category $\cat{M}$ and $\ob a^*$ is a right dual to $\ob a$, then the assignment $h\mapsto(1_{\ob a}\otimes h)\circ(\eta_{\ob a}\otimes1_{\ob b})$ is a bijection  
\begin{equation}\label{dual hom}\Hom_\cat{M}(\ob a^*\otimes\ob b,\ob c)\to\Hom_\cat{M}(\ob b,\ob a\otimes \ob c) \end{equation} 
with inverse $g\mapsto(\ep_{\ob a}\otimes1_{\ob c})\circ(1_{\ob a^*}\otimes g)$.  Similarly, $h\mapsto(1_{\ob c}\otimes \ep_{\ob a})\circ(h\otimes 1_{\ob a})$ is a bijection 
 \begin{equation}\label{dual hom 2}\Hom_\cat{M}(\ob b,\ob c\otimes\ob a^*)\to\Hom_\cat{M}(\ob b\otimes\ob a,\ob c) \end{equation} 
 with inverse $g\mapsto (g\otimes1_{\ob a^*})\circ(1_{\ob b}\otimes \eta_{\ob a})$. 

\begin{lemma}\label{braid inverse} If $\ob a$ is an object of a braided monoidal category $\cat M$ possessing a right dual $\ob{a}^*$, then 
the morphism $$
(\ep_\ob a\otimes 1_\ob a\otimes 1_{\ob a^*})\circ(1_{\ob a^*}\otimes\sigma_{\ob{a,a} }\otimes1_{\ob a^*})\circ(1_{\ob a^*}\otimes1_{\ob a}\otimes \eta_\ob a):\ob a^* \otimes \ob a
\rightarrow \ob a \otimes \ob a^*
$$ 
is invertible.
\end{lemma}

\begin{proof} By Mac Lane's Coherence Theorem (see e.g.~\cite{Mac}) we may assume $\cat M$ is strict. Then we claim that
$(\ep_\ob a\otimes 1_\ob a\otimes 1_{\ob a^*})\circ(1_{\ob a^*}\otimes\sigma_{\ob{a,a} }\otimes1_{\ob a^*})\circ(1_{\ob a^*}\otimes1_{\ob a}\otimes \eta_\ob a)=\sigma_{\ob a, \ob a^*}^{-1}$, which is clearly invertible.
To see this, first notice since $\cat M$ is strict that
$\sigma_{\ob a, \unit}=\sigma_{\ob a, \unit\otimes\unit}=\sigma_{\ob a, \unit}\circ\sigma_{\ob a, \unit}$, which implies $\sigma_{\ob a,\unit}=1_\ob a$.   Hence, 
\begin{align*}
(\ep_\ob a\otimes& 1_\ob a\otimes 1_{\ob a^*})\circ(1_{\ob a^*}\otimes\sigma_{\ob{a,a} }\otimes1_{\ob a^*})\circ(1_{\ob a^*}\otimes1_{\ob a}\otimes \eta_\ob a)\\
&=\sigma_{\ob a, \ob a^*}^{-1}\circ\sigma_{\ob a, \ob a^*}\circ 
(\ep_\ob a\otimes 1_\ob a\otimes 1_{\ob a^*})\circ(1_{\ob a^*}\otimes\sigma_{\ob{a,a} }\otimes1_{\ob a^*})\circ(1_{\ob a^*}\otimes1_{\ob a}\otimes \eta_\ob a)\\
&=\sigma_{\ob a, \ob a^*}^{-1}\circ(\ep_\ob a\otimes 1_{\ob a^*}\otimes 1_{\ob a})\circ(1_{\ob a^*}\otimes1_{\ob{a}}\otimes\sigma_{\ob a, \ob a^*})\circ(1_{\ob a^*}\otimes\sigma_{\ob{a,a} }\otimes1_{\ob a^*})\circ(1_{\ob a^*}\otimes1_{\ob a}\otimes \eta_\ob a) \\
&=\sigma_{\ob a, \ob a^*}^{-1}\circ(\ep_\ob a\otimes 1_{\ob a^*}\otimes 1_{\ob a})\circ(1_{\ob a^*}\otimes\sigma_{\ob a,\ob a\otimes \ob a^*})\circ(1_{\ob a^*}\otimes1_{\ob a}\otimes \eta_\ob a) \\
&=\sigma_{\ob a, \ob a^*}^{-1}\circ(\ep_\ob a\otimes 1_{\ob a^*}\otimes 1_{\ob a})\circ(1_{\ob a^*}\otimes \eta_\ob a\otimes1_{\ob a}) \circ(1_{\ob a^*}\otimes\sigma_{\ob a,\unit})\\
&=\sigma_{\ob a, \ob a^*}^{-1}.
\end{align*}
The lemma is proved.
\end{proof}

\subsection{Gradings and filtrations}\label{filters}
By a {\em graded $\k$-module}, we mean a 
$\k$-module $V$ equipped with a $\k$-module decomposition
$V = \bigoplus_{i \in \Z} V_i$.
If $V$ and $W$ are two graded $\k$-modules,
we write $\Hom_\k(V,W)_i$ for the space of all linear maps
from $V$ to $W$
that are {\em homogeneous of degree $i$}, i.e. they send $V_j$ into $W_{i+j}$
for each $j \in \Z$.
We write $\k\operatorname{-gmod}$ for the category of all
graded $\k$-modules with $$
\Hom_{\k\operatorname{-gmod}}(V, W) := 
\bigoplus_{i \in \Z} \Hom_\k(V,W)_i.
$$
There is a natural monoidal structure 
on this category defined so that $(V \otimes W)_i := \bigoplus_{j \in \Z}V_j \otimes W_{i-j}$.

By a {\em filtered $\k$-module}
we mean a $\k$-module $V$ equipped with a $\k$-module filtration
$\cdots \subseteq V_{\leq i} \subseteq V_{\leq i+1} \subseteq \cdots$
such that $\bigcap_{i \in \Z} V_{\leq i} = 0$ and
$\bigcup_{i \in \Z} V_{\leq i} = V$.
If $V$ and $W$ are two filtered $\k$-modules,
we write $\Hom_\k(V, W)_{\leq i}$ for the space of all $\k$-module
homomorphisms that are of {\em filtered degree $i$}, i.e. they send
$V_{\leq j}$ into $W_{\leq i+j}$ for each $j$.
We write $\k\operatorname{-fmod}$ for the category of all
filtered $\k$-modules with $$
\Hom_{\k\operatorname{-fmod}}(V, W)
:= \bigcup_{i \in \Z} \Hom_{\k}(V, W)_{\leq i}.
$$
Again this category has a natural monoidal structure.
If $V$ is a filtered $\k$-module,
the 
{\em associated graded module} 
$\gr V$ is $\bigoplus_{i \in \Z} V_{\leq i} / V_{\leq i-1}$;
then for filtered $V$ and $W$
each $f \in \Hom_\k(V, W)_{\leq i}$ 
induces $\gr_i f \in \Hom_\k(\gr V,\gr W)_i$
in an obvious way.
Conversely, if $V$ is a graded $\k$-module, it can naturally be viewed as a filtered $\k$-module by setting $V_{\leq i} := \bigoplus_{j \leq i} V_j$;
the associated graded module $\gr V$ to this 
is naturally identified with the original graded module $V$.
 
By a {\em graded category}, we mean a category enriched in the monoidal
category 
consisting of all 
graded $\k$-modules with morphisms being the 
homogeneous linear maps of degree zero.
In other words,
$\cat{C}$ is graded if each $\Hom$ space is graded
in a way that is compatible with composition.
A {\em graded functor} $F:\cat{C}\to\cat{D}$ between graded categories is a functor that maps $\Hom_\cat{C}(\ob a, \ob b)_i$ to $\Hom_\cat{D}(F(\ob a), F(\ob b))_i$ for each pair of objects $\ob a$, $\ob b$ in $\cat{C}$ and each $i\in\Z$. 
 A {\em graded monoidal category} is a monoidal category that is graded
in such a way that $\deg(g\otimes h)=\deg(g)+\deg(h)$ whenever $g$ and $h$ are homogenous.


A \emph{filtered category} is a category enriched in the monoidal category consisting of all filtered $\k$-modules
with morphisms being the linear maps of filtered degree zero.
There is a notion of a {\em filtered functor} between filtered categories,
and of a {\em filtered monoidal category}. 
Given a filtered category $\cat{C}$, the associated graded category
$\gr\cat{C}$ is the graded category with the same objects as $\cat{C}$, 
morphisms defined by setting   $\Hom_{\gr\cat{C}}(\ob a, \ob b)_i:=\Hom_{\cat{C}}(\ob a, \ob b)_{\leq i} /\Hom_{\cat{C}}(\ob a, \ob b)_{\leq i-1}$ for each $i\in\Z$,
and the obvious composition law induced by the composition on $\mathcal C$.
Given a filtered functor $F:\cat{C}\to\cat{D}$, we write $\gr F:\gr\cat{C}\to\gr\cat{D}$ for the graded functor induced by $F$ in the obvious way.  

The categories $\k\text{-gmod}$ and $\k\text{-fmod}$
give examples of graded and filtered monoidal categories, respectively.
The associated graded category $\gr (\k\operatorname{-fmod})$
is {\em not} the same as $\k\operatorname{-gmod}$.
However there is a faithful functor 
\begin{equation}\label{Gfunc}
G:\gr (\k\text{-fmod}) \rightarrow \k\text{-gmod}
\end{equation}
which sends a filtered module $V$ to the associated graded module $\gr V$,
and a morphism $f +\Hom_\k(V,W)_{\leq i-1} 
\in \Hom_\k(V,W)_{\leq i} / \Hom_\k(V,W)_{\leq i-1}$
to $\gr_i f$.

\subsection{Generators and relations}\label{g-r}    Let $\cat M$ be a monoidal category.  
Suppose we are given a presentation of $\cat M$ as a $\k$-linear monoidal category.  Let $G$ denote the set of all generating morphisms, and $R$ denote the  relations.
On a couple occasions it will be useful to forget the monoidal structure, and describe $\cat M$ as merely a $\k$-linear category via generators and relations.  To do so, first note that 
\begin{equation}
\widehat{G} := \{1_\ob a\otimes g\otimes 1_\ob b\:|\:g\in G, \ob{a, b}\in\text{ob}(\cat M)\}
\end{equation}
is a generating set of morphisms for $\cat M$ as a $\k$-linear category. To obtain a full set of relations, every relation in $R$ can be written in the form \begin{equation}\label{R relation}\sum_{i\in I}\lambda_ih_1^{(i)}\circ\cdots\circ h^{(i)}_{n_i}=0\end{equation} where $I$ is some finite set, 
$\lambda_i\in\k$ and $h^{(i)}_j\in \widehat{G}$.  Hence, for each $\ob{a, b}\in\text{ob}(\cat M)$ we have the relation
\begin{equation}\label{1 relations}\sum_{i\in I}\lambda_i(1_\ob a\otimes h_1^{(i)}\otimes 1_\ob b)\circ\cdots\circ (1_\ob a\otimes h^{(i)}_{n_i}\otimes 1_\ob b)=0.\end{equation} 
The collection of relations (\ref{1 relations}) for all $\ob{a, b}\in\text{ob}(\cat M)$ as (\ref{R relation}) ranges over all relations in $R$ does {\em not} form a full set of relations for $\cat M$ as a $\k$-linear category.  
However, it is an easy exercise to show that a full set $\widehat{R}$ of relations can be obtained from this by adding the following 
\emph{commuting relations}
\begin{equation}\label{comm relations} (1_{\ob a}\otimes g\otimes 1_{\ob c\otimes \ob d^{(2)}\otimes \ob e})\circ(1_{\ob a\otimes\ob b^{(1)}\otimes\ob c}\otimes h\otimes 1_{\ob e})=(1_{\ob a\otimes\ob b^{(2)}\otimes\ob c}\otimes h\otimes 1_{\ob e})\circ(1_{\ob a}\otimes g\otimes 1_{\ob c\otimes \ob d^{(1)}\otimes \ob e})
\end{equation} 
for all objects $\ob{a, c, e}$ in $\cat M$ and all morphisms $g:\ob b^{(1)}\to\ob b^{(2)}$, $h:\ob d^{(1)}\to\ob d^{(2)}$ in $G$.  

\section{Oriented Brauer Categories}\label{obc}

\subsection{\boldmath Proof of Theorem~\ref{thm1}}\label{dcat}
We are just going to translate some existing results from the literature
which are phrased in the context of ribbon categories.
There are several
different presentations for ribbon tangles, going back to \cite{Tur89} and \cite{Yet}.  It will be convenient for us to refer to the generators and relations found in \cite{Tur}, so our terminology for ribbon tangles will be consistent with \emph{loc.~cit.}. Consider the category of colored ribbon tangles defined in \cite[\S2.3]{Tur}.  Fix a band color, and let $\Rib$ denote the subcategory consisting of all ribbon tangles whose bands have that color.  Finally, let $\Rib_\k$ denote the $\k$-linearization of $\Rib$, i.e. the category with the same objects as $\Rib$ and morphisms that
are formal $\k$-linear combinations of morphisms in $\Rib$.

There is an obvious functor $\Rib_\k\to \OB$ which maps a ribbon tangle to the oriented Brauer diagram with bubbles obtained by replacing each band with a strand (forgetting the number of twists in each band) and projecting onto the plane (forgetting all over/under crossing information).  It is easy to see that this functor is full. In fact, with this functor in mind, it is apparent that one could define oriented Brauer diagrams with bubbles as certain equivalence classes of ribbon tangles.  

A set of generators and relations for the category of colored ribbon tangles as a monoidal category is given in \cite[Lemma 3.1.1 and Lemma 3.3]{Tur}.
Restricting to bands of just one color, 
one easily extracts from this a
presentation for the category $\Rib_\k$ as a $\k$-linear  monoidal category.   Including the obvious additional relations to forget twists in bands and over/under crossing information,
 we are left with a presentation of $\OB$ given by the generators and relations prescribed by Theorem~\ref{thm1} along with the additional relation \begin{equation*}
 (d\up)\circ(\down s)\circ(t\up)\circ(c\up)=\up,
\end{equation*}
where $t$ is the inverse of the morphism in (\ref{I}).
To finish the proof, we must show that this extra relation is a consequence of the other relations. In fact, the following computation uses only (\ref{A down}) and the definition of $t$:
\begin{align*}
 (d\up)\circ(\down s)\circ(t\up)\circ(c\up)
&=(d\up)\circ(\down s)\circ(\down\up\up d)\circ(\down\up c\up)\circ(t\up)\circ
(c\up) \\
&=(\up d)\circ(d\up\down\up)\circ(\down s\down\up)\circ(\down\up c\up)\circ(t\up)\circ
(c\up)\\
&=(\up d)\circ(c\up)=\up.
\end{align*}
This completes the proof of Theorem~\ref{thm1}.

\begin{corollary}\label{OB functor} 
Suppose $\cat M$ is a symmetric monoidal category
and $\ob a$ is an object of $\cat M$ possessing a right dual $\ob a^*$.
Then the assignment on objects $\up\mapsto\ob a$, $\down\mapsto\ob a^*$ and morphisms $c\mapsto\eta_\ob a$, $d\mapsto\ep_\ob a$, $s\mapsto\sigma_{\ob a, \ob a}$ prescribes a tensor functor $\OB\to\cat M$.
\end{corollary}

\begin{proof} By Theorem~\ref{thm1} it suffices to verify that the images of (\ref{A down})-(\ref{I}) hold in $\cat M$.  By Mac Lane's Coherence Theorem we may assume $\cat M$ is a strict monoidal category.  
The images of (\ref{A down}), (\ref{A up}), and (\ref{S}) hold by the definitions of duals and symmetric braidings.  The image of (\ref{B}) follows from the Yang-Baxter equation (\ref{Yang-Baxter}).  Finally, the image of (\ref{I}) follows from Lemma \ref{braid inverse}.
\end{proof} 

\subsection{Reversed orientations}
It is obvious from the diagrammatic definition of $\OB$ that there is a self-inverse tensor functor
\begin{equation}\label{rev}
\OB \to \OB, \quad g \mapsto g'
\end{equation} 
defined by switching the objects $\up$ and $\down$
and reversing the orientations of all the strands in a diagram.
Applying this
to the generators and relations 
from Theorem~\ref{thm1}
yields an alternative presentation of $\OB$ with the following generating morphisms:   
$$\includegraphics{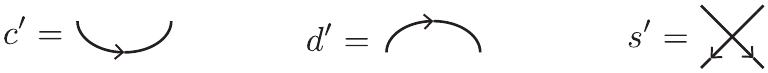}$$  
In terms of the original generating morphisms,
these alternative generators are given explicitly by 
\begin{equation}\label{primes}  c'=t\circ c,\quad d'=d\circ t,\quad s'=(d\down\down)\circ
(\down d\up\down\down)\circ(\down \down s\down\down)\circ(\down\down\up c\down)
\circ(\down\down c).
\end{equation}
The relations that they satisfy are obtained from
(\ref{A down})--(\ref{I}) by reversing all orientations:
\begin{align}\label{R1}
(\down d')\circ(c'\down)&=\down,\\
(d'\up)\circ(\up c')&=\up,\\
{s'}^2&=\down\down,\\
(\down s')\circ(s'\down)\circ(\down s')&=(s'\down)\circ(\down s')\circ (s'\down),\\
\label{R2}
(d'\down\up)\circ(\up s'\up)\circ(\up\down c')&\text{ is
  invertible.}
 \end{align}
We remark that it is not at all trivial to derive
these 
primed relations directly from (\ref{A down})--(\ref{I}) and (\ref{primes}) (i.e. without going via diagrams). 
This nicely illustrates the substance of Turaev's results exploited in the proof of Theorem \ref{thm1}.

We use the same notation $c', d', s'$ for the images in $\AOB$ of these
morphisms under the
functor $\OB \rightarrow \AOB$;
explicitly one can 
take (\ref{primes}) as the definition of $c', d', s'$ in $\AOB$. Also recall 
the morphism $x':= (d \down) \circ (\down x \down) \circ (\down c) \in \End_{\AOB}(\down)$ from the introduction.
Note that
\begin{equation}\label{more}
(\down x') \circ s' = s' \circ (x'\down) - \down\down.
\end{equation}
The sign here is different from in (\ref{AX}), making it clear that reversing orientations does {\em not} in general define a self-equivalence of $\AOB$.
Nevertheless,
the elements $c',d',s',x'$ give an alternative set of generators for $\AOB$
subject only to the relations
(\ref{R1})--(\ref{more}).

\subsection{Jucys-Murphy morphisms}\label{JM}  Given a word 
$\ob a\in\wrd$ of length $k$, and distinct integers $1 \leq p, q\leq k$, 
let $(p,q)^{\ob a} \in \End_{\OB}(\ob a)$
be the morphism whose diagram is either the crossing of the $p$th and $q$th strands
if $\ob a_p = \ob a_q$, or minus the cap-cup pair joining the $p$th and $q$th letters if $\ob a_p \neq \ob a_q$, with 
all other strands going straight through.
For example, if $\ob a=\up\down\up\down\down$ then
$$\includegraphics{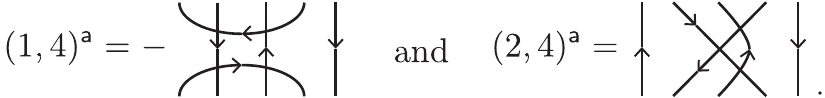}$$
Then for fixed $\ob a$ and $p$, 
the corresponding \emph{Jucys-Murphy morphism} is defined to  be 
\begin{equation*}
JM^{\ob a}_p:=\sum_{0<q<p}(p,q)^\ob a\in\End_{\OB}(\ob a).
\end{equation*} 
In particular, $JM^{\ob a}_1=0$ for all $\ob a$.

\begin{lemma}\label{JM props} Suppose $\ob a, \ob b,\ob c, \ob d\in\wrd$ and let $k,l$ denote the lengths of the words $\ob a,\ob c$ respectively.
\begin{itemize}
\item[(i)] $(g\up h)\circ JM_{k+1}^{\ob a\up\ob b}=JM_{l+1}^{\ob c\up\ob d}\circ(g\up h)$ for any morphisms $g:\ob a\to\ob c$ and $h:\ob b\to\ob d$.  
\item[(ii)] $JM_{k+2}^{\ob a\up\up\ob b}\circ(\ob a\,s\,\ob b)=(\ob a\,s\,\ob b)\circ JM_{k+1}^{\ob a\up\up\ob b}+\ob a\up\up\ob b$.
\end{itemize}
\end{lemma}

\begin{proof}
We leave this as an exercise for the reader. Note by Theorem \ref{thm1} that one only needs to verify part (i) for those $g$ which are tensor products of identity morphisms and a single $c, d, s$ or $t$.
\end{proof}

The following theorem is an easy consequence of Theorems~\ref{thm2} and \ref{thm3} from the introduction, and actually it will never be needed in our
proofs of those results.
Nevertheless we include a self-contained proof right away since we found it to be quite instructive.

\begin{theorem}
\label{level one} Suppose that $f(u)=u-m \in \k[u]$ is monic of degree one.
Then the
functor $\OB \rightarrow \COB^f$
defined as the composite
first of the 
functor $\OB \rightarrow \AOB$
then the quotient functor $\AOB \to \COB^f$
is an isomorphism.
\end{theorem}

\begin{proof}
It is easy to see that the morphism spaces in $\COB^f$ are spanned by all oriented Brauer diagrams with bubbles (no dots). 
For example:
$$\includegraphics{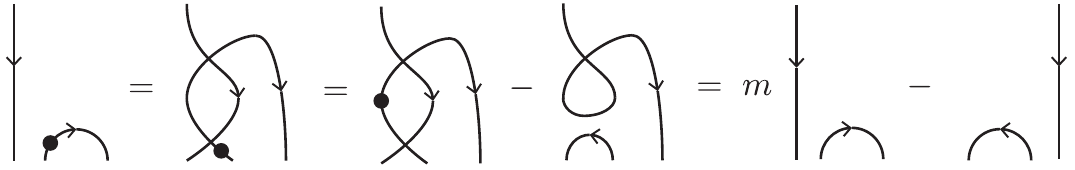}$$
Hence the given composite functor $G:\OB \rightarrow \COB^f$ is full.

To show that it is faithful, we will construct a functor $F:\COB^f\to \OB$ such that
$F \circ G = \operatorname{Id}$. This depends on a description of the $\k$-linear category $\COB^f$ via generators and relations.  
To get this, we first
forget the monoidal structure to get a description of $\AOB$ as a $\k$-linear category via generators and relations as explained in \S\ref{g-r}.  Doing so yields the following set of generating morphisms for $\AOB$: $\ob a\,c\,\ob b, \ob a\,d\,\ob b, \ob a\,s\,\ob b, \ob a\,t\,\ob b$, $\ob a\,x\,\ob b$ where $\ob a$ and $\ob b$ range over all words in $\wrd$.  This set also generates the morphisms in $\COB^f$.  A full set of relations for $\COB^f$ is obtained from the relations for the $\k$-linear category $\AOB$ by adding 
\begin{equation}\label{level one relation}
x\,\ob a=m(\up\ob a)\quad\text{for all }\ob{a}\in\wrd.
\end{equation}
Now, it is clear how to define the desired functor $F$  on all generators except those involving $x$; for them we
set $F(\ob a\,x\,\ob b):=JM_{k+1}^{\ob a\up\ob b} +m(\ob a\up\ob b)$ where $k$ denotes the length of $\ob a$.  To see that this is well defined we must
show that $F$ preserves the relations for $\COB^f$. Again this is obvious for relations not involving $x$.  Then part (i) of Lemma \ref{JM props} guarantees that $F$ preserves all commuting relations involving $x$.  
Also $F$ preserves (\ref{level one relation}) 
since $JM_1^{\up\ob a}=0$ for all $\ob a\in\wrd$.  The only relations left to check are 
\begin{equation}\label{AXdumb} (\ob a\up x\,\ob b)\circ (\ob a\,s\,\ob b)=(\ob a\,s\,\ob b)\circ(\ob a\,x\up\ob b)+\ob a\up\up\ob b\quad\text{for all }\ob{a,b}\in\wrd.
\end{equation}  
These follow by Lemma~\ref{JM props}(ii).
\end{proof}

\begin{corollary}
The functor $\OB \rightarrow \AOB$ is faithful.
\end{corollary}

\subsection{Associated graded categories}\label{GOB}
There is one more relevant category of oriented Brauer diagrams which,
like $\OB$, is most easily defined 
diagrammatically from the outset.
The {\em graded oriented Brauer category} $\GOB$ is the monoidal category
with objects $\wrd$ and morphisms
$\Hom_{\GOB}(\ob a, \ob b)$
consisting of $\k$-linear combinations of 
equivalence classes of normally ordered 
dotted oriented Brauer diagrams with bubbles of type $\ob a \rightarrow \ob b$.
The rules for vertical and horizontal composition are by stacking diagrams as usual. Both 
may produce 
diagrams that are no longer normally ordered;
to convert the resulting diagrams into normally ordered ones, it is now
permissible to move dots past crossings; so one
simply 
translates all (possibly dotted) bubbles into clockwise ones at the left edge,
and slides all other
dots past crossings so that they are on outward-pointing segments.
For example, the first dotted oriented Brauer diagram displayed in the introduction gets transformed directly in this way into the second (normally ordered) one.
It is quite obvious that the compositions defined in this way are associative,
and make $\GOB$ into a well-defined $\k$-linear monoidal category.
It is also rigid and symmetric, with duals and braiding defined 
in the same way as we did for the category $\OB$ in the introduction.

\begin{theorem}\label{GOB  g-r} As a $\k$-linear  monoidal category, $\GOB$ is generated by objects $\up, \down$ and morphisms $c, d, s, x$ subject only to the  relations (\ref{A down})--(\ref{I}) plus
\begin{equation}\label{GX}
(\up x)\circ s=s\circ(x\up).
\end{equation}
\end{theorem}

\begin{proof} Let $\cat{C}$ denote the category defined by the generators and relations in the statement of the theorem.  Identifying 
$c,d,s,t,x$ with the morphisms in $\GOB$ associated  to  
the (by now) familiar diagrams, it is clear that
the relations (\ref{A down})--(\ref{I}) and (\ref{GX}) hold in $\GOB$.
Hence we have a functor $F:\cat{C}\to \GOB$.  This functor is bijective on objects. As noted already in the introduction,
every dotted oriented Brauer diagram with bubbles is equivalent to a vertical 
composition of diagrams of the form 
$\ob a \, s\, \ob b$,
$\ob a \, t\, \ob b$,
$\ob a \, c\, \ob b$,
$\ob a \, d\, \ob b$,
$\ob a \, x\, \ob b$.
Hence $F$ is full.
Finally to see that $F$ is an isomorphism it remains to show that we have enough relations in $\cat{C}$.
In view of Theorem~\ref{thm1}, this amounts to checking
that there are enough relations in $\cat{C}$ to move dots past crossings.
This follows by (\ref{GX}) and the relation obtained from that by composing with $s$ on the top and bottom.
\end{proof}

\begin{corollary}\label{GOB functor} Suppose $\cat M$ is a symmetric monoidal category, $\ob a$ is an object of $\cat M$ possessing a right dual $\ob a^*$,
and $g\in\End_\cat M(\ob a)$.
Then the assignment on objects $\up\mapsto\ob a$, $\down\mapsto\ob a^*$ and morphisms $c\mapsto\eta_\ob a$, $d\mapsto\ep_\ob a$, $s\mapsto\sigma_{\ob a, \ob a}$, $x\mapsto g$ prescribes a tensor functor $\GOB\to\cat M$.
\end{corollary}

\begin{proof} 
By Theorem~\ref{GOB g-r} and Corollary~\ref{OB functor}, it suffices to show the image of (\ref{GX}) holds in $\cat M$, which follows from the fact that $\sigma$ is a natural transformation.  
\end{proof}

Let $\cat{F}$ denote the free $\k$-linear monoidal category generated by objects $\up$ and $\down$ and morphisms $c:\varnothing\to\up\down$, $d:\down\up\to\varnothing$, $s:\up\up\to\up\up$, 
$t:\up\down\to\down\up$
and  $x:\up\to\up$.  We put a monoidal grading  on $\cat{F}$ by setting $c,d,s$ and $t$ in degree 0 and $x$ in degree 1 (see \S\ref{filters}). Since the relations (\ref{A down})--(\ref{I}) and (\ref{GX}) are all homogenous with respect to this grading, Theorem \ref{GOB g-r} implies that $\GOB$ inherits a grading from $\cat{F}$ in which 
 the degree of each dotted diagram is the number of dots in that diagram.  On the other hand, since (\ref{AX}) is not homogeneous, $\AOB$ merely inherits a filtration from $\cat{F}$.  Since (\ref{AX}) reduces to (\ref{GX}) in the associated graded category $\grAOB$, there is a graded tensor functor 
\begin{equation}\label{GOB to gr AOB}
\Theta:\GOB\to \grAOB.
\end{equation}
The filtered degree $i$ part of each morphism space in $\AOB$ is spanned by the morphisms arising from
normally ordered dotted oriented Brauer diagrams with bubbles having at most $i$ dots. This makes it clear that the functor $\Theta$ is full.
Since the normally ordered diagrams are linearly independent in $\GOB$
by its definition, we see further that 
Theorem~\ref{thm2} from the introduction is equivalent to the following.

\begin{theorem}\label{12prime}
The functor $\Theta:\GOB \to \grAOB$ is an isomorphism.
\end{theorem}

There is also a natural candidate for the associated graded category
to $\COB^f$, assuming now that $f(u) \in \k[u]$ is monic of degree $\ell$.
Consider the tensor ideal in $\GOB$ generated by $x^\ell$.  Its morphism spaces are spanned by diagrams in which at least one strand has at least $ \ell$ dots.  The {\em graded oriented Brauer category of level $\ell$} is the quotient $\TOB^\ell$ of $\GOB$ by this tensor ideal.  Normally ordered diagrams with fewer than $\ell$ dots on each strand give a basis for each Hom space in $\TOB^\ell$.
Moreover, since $\GOB$ is symmetric, $\TOB^\ell$ inherits the structure of a rigid symmetric monoidal category from $\GOB$ (see \S\ref{ideals}).
The truncation $\TOB^\ell$ inherits a grading from $\GOB$,
while
$\COB^f$ inherits a filtration from $\AOB$.  
Moreover the functor $\Theta$ induces a functor 
 \begin{equation}\label{G to grOB_f}\Theta^f: \TOB^\ell\to \grOB^f.
\end{equation}
Like in the previous paragraph, this functor is full,
and Theorem~\ref{thm3} from the introduction is equivalent to the following.

\begin{theorem}\label{15prime}
The functor $\Theta^f:\TOB^\ell \to \grOB^f$ is an isomorphism.
\end{theorem}

Finally fix also some scalars $\delta=\delta_1,\dots,\delta_\ell \in \k$.
The filtration on $\COB^f$ induces a filtration on the
specialized category $\COB^f(\delta_1,\dots,\delta_\ell)$.
Moreover the functor $\Theta^f$ specializes to a full functor
\begin{equation}\label{Actual}
\Theta^f(\delta_1,\dots,\delta_\ell):\TOB^\ell(\delta) \rightarrow \grOB^f(\delta_1,\dots,\delta_\ell),
\end{equation}
where $\TOB^\ell(\delta)$ is the graded rigid symmetric monoidal category
obtained from $\TOB^\ell$ by evaluating the undotted bubble of degree zero 
at $\delta$ and all of the dotted bubbles of 
positive degree at zero.
Given the truth of Theorem~\ref{15prime}, it follows easily that this specialized functor is also an isomorphism. In the remainder of the article, 
we are going to argue in the opposite direction,
first showing that the 
specialized functor is an isomorphism on sufficiently many Hom spaces
assuming that 
$\k$ is an algebraically closed field of characteristic zero, then deducing 
Theorem~\ref{15prime} and the other main theorems by some density/base change arguments.

\section{Representations}\label{repn}

\subsection{Combinatorial data}\label{data}
In this section we are going to define various representations.
These depend on some data which will be fixed throughout the section.
To start with let $\ell \geq 1$ and pick scalars $m_1,\dots,m_\ell \in \k$.
Next let $\lambda=(\lambda_1,\ldots,\lambda_\ell)$ be a unimodular sequence of positive integers, i.e. we have that $1 \leq \lambda_1 \leq \cdots \leq \lambda_k \geq \cdots \geq \lambda_\ell \geq 1$
for some $1 \leq k \leq \ell$.
Set $n:=\lambda_1+\cdots+\lambda_\ell$. We identify $\lambda$ with a pyramid-like array of boxes with $\lambda_j$ counting the number of boxes in the $j$th column (numbering columns from left to right).  For example, the pyramid for $\lambda=(2,3,2,1,1)$ is the following:
\begin{equation}\label{neeg}
{\begin{picture}(90, 31)%
\put(0,-15){\line(1,0){75}}
\put(0,0){\line(1,0){75}}
\put(0,15){\line(1,0){45}}
\put(15,30){\line(1,0){15}}
\put(0,-15){\line(0,1){30}}
\put(15,-15){\line(0,1){45}}
\put(30,-15){\line(0,1){45}}
\put(45,-15){\line(0,1){30}}
\put(60,-15){\line(0,1){15}}
\put(75,-15){\line(0,1){15}}
\put(7,-7){\makebox(0,0){5}}
\put(22,-7){\makebox(0,0){6}}
\put(37,-7){\makebox(0,0){7}}
\put(52,-7){\makebox(0,0){8}}
\put(67,-7){\makebox(0,0){9}}
\put(7,8){\makebox(0,0){2}}
\put(22,8){\makebox(0,0){3}}
\put(37,8){\makebox(0,0){4}}
\put(22,23){\makebox(0,0){1}}
\put(79,-8){\makebox(0,0){.}}
\end{picture}}
\vspace{4mm}
\end{equation}
As pictured above, we number the boxes in the pyramid $1,\ldots,n$ along the rows starting at the top.
We write $\col(i)$ for the column number of the $i$th box.   For instance, in the pyramid pictured above $\col(6)=2$.
Finally 
introduce the monic polynomials of degree $\ell$:
\begin{align}
f(u) &:= (u-m_1)\cdots (u-m_\ell),\\
f'(u) &:= (u+\lambda_1-m_1)\cdots(u+\lambda_\ell-m_\ell).
\end{align}
To these polynomials, we associate 
scalars $\delta_1,\delta_2,\dots \in \k$ via the 
generating function identity (\ref{deltas}).
Explicitly,
\begin{equation}\label{deltadef}
\delta_k = 
\sum_{i+j=k}h_i(m_1,...,m_\ell) e_j (\lambda_1-m_1,...,\lambda_\ell - m_\ell) 
\end{equation}
where $h_i$ and $e_j$ denote complete and elementary symmetric polynomials,
respectively.
Note in particular that $\delta := \delta_1$ is equal to the integer $n$.

\subsection{\boldmath A representation of $\OB(\delta)$}\label{OB action}  
Let $V$ be the free $\k$-module on basis $v_1,\ldots,v_n$, and
$f_1,\ldots,f_n$ be the dual basis for $V^\ast:=\Hom_\k(V,\k)$. 
It will be convenient to set
$V^{\up} := V$, 
$V^{\down} := V^*$,
$v_i^\up:=v_i$ and $v_i^\down:=f_i$.
Then, for each $\ob a=\ob a_1\cdots\ob a_k\in\wrd$, the $\k$-module 
\begin{equation}
V(\ob a)
:= V^{\ob a_1} \otimes\cdots\otimes V^{\ob a_k}
\end{equation}
is free with a basis consisting of the monomials $v^{\ob a}_{\bd i}:=v_{i_1}^{\ob a_1}\otimes\cdots\otimes v_{i_k}^{\ob a_k}$ for all $k$-tuples ${\bd i}=(i_1,\ldots,i_k)$ of integers with $1\leq i_1,\ldots,i_k\leq n$.  

Note that
$V^\ast$ is a right dual of $V$ in the sense of \S\ref{duality} with $\ep_V:V^\ast\otimes V\to\k$ given by evaluation $f\otimes v\mapsto f(v)$ and $\eta_V:\k\to V\otimes V^\ast$ given by $1\mapsto\sum_{i=1}^nv_i\otimes f_i$.  
Applying Corollary~\ref{OB functor} to this data, we get a  tensor functor 
\begin{equation}\label{Psi}
\Psi:\OB\to\k\text{-mod}
\end{equation} 
which sends 
object $\ob a$ to
$V(\ob a)$.
Given a morphism $g:\ob a\to\ob b$ in $\OB$, 
$\Psi(g)$ is a linear map $V(\ob a)\to V(\ob b)$; we often denote the
image of $u \in V(\ob a)$ under this map simply by $gu$. 
For example, if $\ob a=\up\up\down\down$, then using the notation set up in \S\ref{JM} we have that
$$(1,2)^{\ob a} v^{\ob a}_{\bd i}=v^{\ob a}_{(i_2,i_1,i_3,i_4)}\quad\text{and}\quad(2,4)^{\ob a} v^{\ob a}_{\bd i}=\begin{cases}
-\sum_{j=1}^nv^{\ob a}_{(i_1,j,i_3,j)} &\text{if }i_2=i_4;\\
0 & \text{if }i_2\not=i_4.
\end{cases}
$$  
Finally we observe that $\Psi$ maps the bubble $\Delta$ to the scalar $\delta$ (which we recall is equal to
the dimension $n=\dim V$), hence it factors through the quotient $\OB(\delta)$
to induce a tensor functor
\begin{equation}
\Psi(\delta):\OB(\delta)\to\k\text{-mod}.
\end{equation}

\subsection{Modified transpositions}\label{mod transp}  Fix a word $\ob a\in\wrd$ of length $k$ and distinct integers $1 \leq p,q\leq k$ such that $\ob a_p=\up$.  Define scalars $\beta_{\bd i,\bd j}\in\k$ from the equation 
$(p,q)^\ob a v^\ob a_{\bd j}=\sum_{\bd i}\beta_{\bd i, \bd j}v^\ob a_{\bd i}$.
Then define a linear map 
\begin{equation}\label{piano}
(p,q)^\ob a_\lambda:V(\ob a)\to V(\ob a)
\end{equation} 
by setting $(p,q)^\ob a_\lambda v^\ob a_{\bd j}=\sum_{\bd i}\gamma_{\bd i, \bd j}v^\ob a_{\bd i}$
where \begin{equation*}\gamma_{\bd i, \bd j}=\begin{cases}
\beta_{\bd i, \bd j}& \text{if $p>q$ and $\col(i_p)\geq\col(j_p)$};\\
-\beta_{\bd i,\bd j} & \text{if $p<q$ and $\col(i_p)<\col(j_p)$};\\
0 & \text{otherwise}.
\end{cases}\end{equation*}
For example,  the action of $(p,q)^\ob a_{(n)}$ agrees with the action of $(p,q)^\ob a$ when $p > q$, but $(p,q)^\ob a_{(n)}$ acts as zero when $p <q$.  For a more explicit example, let $\lambda$ be the pyramid pictured in \S\ref{data} and take $\ob a=\up\up\down\down$.  Then we have 
\begin{equation*}
(1,2)^\ob a_\lambda v^\ob a_{(7,6,5,6)}= -v^\ob a_{(6,7,5,6)}\quad\text{and}\quad(2,4)^\ob a_\lambda v^\ob a_{(7,6,5,6)}=v^\ob a_{(7,2,5,2)}+v^\ob a_{(7,5,5,5)}.
\end{equation*}
The maps $(p,q)_\lambda^{\ob a}$ just defined will be a key ingredient for defining a representation of $\AOB$ in $\S$\ref{repn AOB}.  The remainder of this subsection is devoted to recording several technical formulae involving $(p,q)_\lambda^{\ob a}$ which will be needed later.  They can all be proved in a straightforward manner using the definition of $(p,q)^{\ob a}_\lambda$.  We leave the details as an exercise for the reader.

\begin{lemma}\label{mod (p,q) prop} Suppose $p$ and $r$ are distinct positive integers and $\ob a\in \wrd$ satisfies $\ob a_p=\ob a_r=\up$.  
\begin{itemize}
\item[(i)] $(p,r)^\ob a_\lambda\circ(r,p)^\ob a_\lambda=0$.
\item[(ii)] $(p,q)^\ob a_\lambda\circ(r,q)^\ob a_\lambda=(r,q)^\ob a_\lambda\circ(p,r)^\ob a_\lambda+(r,p)^\ob a_\lambda\circ(p,q)^\ob a_\lambda$ whenever $a_q=\up$.
\item[(iii)] $(p,q)^\ob a_\lambda\circ(r,q)^\ob a_\lambda+(p,r)^\ob a_\lambda\circ(r,q)^\ob a_\lambda+(p,q)^\ob a_\lambda\circ(r,p)^\ob a_\lambda=0$ whenever $a_q=\down$.
\item[(iv)] $(p,q_1)^\ob a_\lambda\circ(r,q_2)^\ob a_\lambda=(r,q_2)^\ob a_\lambda\circ(p,q_1)^\ob a_\lambda$ whenever $p, r, q_1, q_2$ are pairwise distinct.
\end{itemize}
\end{lemma}

For the next three lemmas, let $\ob a, \ob b\in\wrd$, let $k$ denote the length of $\ob a$, and let $p$ be a positive integer such that the $p$th letter in the word $\ob a\down\down\ob b$ is $\up$.  

\begin{lemma}\label{c and (p,q)}   
~
\begin{itemize} 
\item[(i)] $\Psi(\ob a\,c\,\ob b)\circ(p,q)^{\ob a\ob b}_\lambda=(p,q)^{\ob a\up\down\ob b}_\lambda\circ \Psi(\ob a\,c\,\ob b)$ whenever $p,q\leq k$.
\item[(ii)] $\Psi(\ob a\,c\,\ob b)\circ(p,q-2)^{\ob a\ob b}_\lambda=(p,q)^{\ob a\up\down\ob b}_\lambda\circ \Psi(\ob a\,c\,\ob b)$ whenever $p\leq k$ and $q>k+2$.
\item[(iii)] $\Psi(\ob a\,c\,\ob b)\circ(p-2,q)^{\ob a\ob b}_\lambda=(p,q)^{\ob a\up\down\ob b}_\lambda\circ \Psi(\ob a\,c\,\ob b)$ whenever $p>k+2$ and $q\leq k$.
\item[(iv)] $\Psi(\ob a\,c\,\ob b)\circ(p-2,q-2)^{\ob a\ob b}_\lambda=(p,q)^{\ob a\up\down\ob b}_\lambda\circ \Psi(\ob a\,c\,\ob b)$ whenever $p,q>k+2$.
\item[(v)] $(p,k+1)^{\ob a\up\down\ob b}_\lambda\circ\Psi(\ob a\,c\,\ob b)+(p,k+2)^{\ob a\up\down\ob b}_\lambda\circ\Psi(\ob a\,c\,\ob b)=0$.
\end{itemize}
\end{lemma}

\begin{lemma}\label{d and (p,q)} 
~
\begin{itemize}
\item[(i)]$\Psi(\ob a\,d\,\ob b)\circ(p,q)^{\ob a\down\up\ob b}_\lambda=(p,q)^{\ob a\ob b}_\lambda\circ \Psi(\ob a\,d\,\ob b)$ whenever $p,q\leq k$.
\item[(ii)]$\Psi(\ob a\,d\,\ob b)\circ(p,q)^{\ob a\down\up\ob b}_\lambda=(p,q-2)^{\ob a\ob b}_\lambda\circ \Psi(\ob a\,d\,\ob b)$ whenever $p\leq k$ and $q>k+2$.
\item[(iii)]
$\Psi(\ob a\,d\,\ob b)\circ(p,q)^{\ob a\down\up\ob b}_\lambda=(p-2,q)^{\ob a\ob b}_\lambda\circ \Psi(\ob a\,d\,\ob b)$ whenever $p>k+2$ and $q\leq k$.
\item[(iv)] 
$\Psi(\ob a\,d\,\ob b)\circ(p,q)^{\ob a\down\up\ob b}_\lambda=(p-2,q-2)^{\ob a\ob b}_\lambda\circ \Psi(\ob a\,d\,\ob b)$ whenever $p,q>k+2$.
\item[(v)]
$\Psi(\ob a\,d\,\ob b)\circ(p,k+1)^{\ob a\down\up\ob b}_\lambda+\Psi(\ob a\,d\,\ob b)\circ(p,k+2)^{\ob a\down\up\ob b}_\lambda=0$.
\end{itemize}
\end{lemma}

\begin{lemma}\label{t and (p,q)} 
~
\begin{itemize}
\item[(i)]
$\Psi(\ob a\,s\,\ob b)\circ(p,q)^{\ob a\up\up\ob b}_\lambda=(p,q)^{\ob a\up\up\ob b}_\lambda\circ \Psi(\ob a\,s\,\ob b)$ whenever $q\leq k$ or $q>k+2$.
\item[(ii)]
$\Psi(\ob a\,s\,\ob b)\circ(p,k+1)^{\ob a\up\up\ob b}_\lambda=(p,k+2)^{\ob a\up\up\ob b}_\lambda\circ \Psi(\ob a\,s\,\ob b)$.
\item[(iii)]
$\Psi(\ob a\,s\,\ob b)\circ(p,k+2)^{\ob a\up\up\ob b}_\lambda=(p,k+1)^{\ob a\up\up\ob b}_\lambda\circ \Psi(\ob a\,s\,\ob b)$. 
\item[(iv)]
$\Psi(\ob a\,s\,\ob b)\circ(k+1,q)^{\ob a\up\up\ob b}_\lambda=(k+2,q)^{\ob a\up\up\ob b}_\lambda\circ\Psi(\ob a\,s\,\ob b)$ whenever $q\leq k$ or $q>k+2$.
\item[(v)]
$(k+2,k+1)^{\ob a\up\up\ob b}_\lambda\circ\Psi(\ob a\,s\,\ob b)=\Psi(\ob a\,s\,\ob b)\circ(k+1,k+2)^{\ob a\up\up\ob b}_\lambda+1_{V(\ob a\up\up\ob b)}$.  
\end{itemize}
\end{lemma}

\subsection{\boldmath A graded representation of $\TOB^\ell(\delta)$}\label{thenilpotent}
We now put a $\Z$-grading on the $\k$-modules 
$V$ and $V^*$ 
from $\S$\ref{OB action}
by declaring that $\deg(v_i) = -\col(i)$
and $\deg(f_i) = \col(i)$.
We get an induced grading on $V(\ob a)$
for each $\ob a \in \wrd$.
Let $e \in \End_\k(V)$
be the homogeneous
linear transformation of degree one that maps basis vector $v_i$ to
$v_{i-1}$ if
the $i$th box of the pyramid $\lambda$ is {\em not} the leftmost box in its row,
or to zero otherwise.
Clearly this is nilpotent with Jordan block sizes equal to the
lengths of the rows of the pyramid $\lambda$; in particular $e^\ell = 0$.

Applying Corollary~\ref{GOB functor}
we get a tensor functor 
\begin{equation}
\Phi_\lambda:\GOB \rightarrow \k\text{-gmod}
\end{equation}
sending $\up$ to $V$,
$\down$ to $V^*$,
and $x$ to $e$.
This functor is obviously graded.
Moreover, since $e^\ell = 0$, it factors through the quotient $\TOB^\ell$
to induce 
\begin{equation}
\Phi_\lambda^\ell:\TOB^\ell \rightarrow \k\text{-gmod}.
\end{equation}
Finally we observe that $\Phi_\lambda^\ell$ maps the undotted bubble of degree zero to the scalar $\delta$, and it sends all other dotted bubbles to zero.
Hence it factors again to induce a graded tensor functor
\begin{equation}\label{Cat}
\Phi_\lambda^\ell(\delta): \TOB^\ell(\delta) \rightarrow \k\text{-gmod}.
\end{equation}
In particular, given a word $\ob a = \ob a^{(1)} \up \ob a^{(2)}$
with $\ob a^{(1)}$ of length $(p-1)$, 
this functor sends the morphism 
$\ob a^{(1)} \,x \, \ob a^{(2)}$ to the linear map
\begin{equation}\label{Lp}
e_p:V(\ob a) \rightarrow V(\ob a)
\end{equation}
defined by applying $e$ to the $p$th tensor position.

\begin{lemma}\label{vustinj}
Let $\ob a, \ob b \in \wrd$ be words whose average length
is $\leq \min(\lambda)$.
Then the linear map 
$\Hom_{\TOB^\ell(\delta)}(\ob a, \ob b)
\rightarrow \Hom_{\k}(V(\ob a), V(\ob b))$
induced by $\Phi^\ell_\lambda(\delta)$ is injective.
\end{lemma}

\begin{proof}
We first prove this in the special case that
$\ob a = \ob b = \up^m$ for some $m \leq \min(\lambda)$.
Then it is obvious that the algebra 
$\End_{\GOB^\ell(\delta)}(\ob a)$ 
is the smash product $\left(\k[x_1,\dots,x_m] / (x_1^\ell,\dots,x_m^\ell) \right) \rtimes S_m$,
where $x_i := \up^{i-1} x \up^{m-i}$
and $S_m$ is the symmetric group generated by the
transpositions $s_j := \up^{j-1} s \up^{m-j-1}$.
In particular it has a basis consisting of monomials $x_1^{k_1} \cdots x_m^{k_m} w$
for $0 \leq k_1,\dots,k_m < \ell$ and $w \in S_m$.
Let $\bd i = (i_1,\dots,i_m)$ be a tuple consisting
of the numbers of the
boxes in the rightmost column and the bottom $m$ rows of the pyramid $\lambda$.
The assumption on $m$ means that each of these 
rows contains $\ell$ boxes.
Then $$\Phi^\ell_\lambda(\delta) (x_1^{k_1} \cdots x_m^{k_m} w) (v_{\bd i}^{\ob a})
= e_1^{k_1} \cdots e_m^{k_m} (v_{w(\bd i)}^{\ob a})
$$
where $w(\bd i) := (i_{w^{-1}(1)},\dots,i_{w^{-1}(m)})$.
These vectors for all $0 \leq k_1,\dots,k_m < \ell$ and $w \in S_m$
are clearly linearly independent. The desired injectivity follows.

Now for the general case, take any words $\ob a, \ob b \in \wrd$.
Suppose that $\ob a$ involves $p$ $\up$'s and $q$ $\down$'s,
while $\ob b$ involves $r$ $\up$'s and $s$ $\down$'s.
We may assume that $p+s=r+q$, 
since otherwise $\Hom_{\TOB^\ell(\delta)}(\ob a, \ob b) = 0$.
This number is the average length $m$ of the words $\ob a$ and $\ob b$,
so $m \leq \min(\lambda)$.
Consider the following diagram: 
\begin{equation*}
\begin{CD}
\Hom_{\TOB^\ell(\delta)}(\ob a, \ob b)&@>>>&
\Hom_{\TOB^\ell(\delta)}(\down^q \up^p, \up^r \down^s)
&@>>>&
\End_{\TOB^\ell(\delta)}(\up^{m})\\
@V\Phi^\ell_\lambda(\delta)VV&
&@V\Phi^\ell_\lambda(\delta)VV&
&@VV\Phi^\ell_\lambda(\delta) V\\
\Hom_{\k}(V(\ob a), V(\ob b))&@>>>&
\Hom_{\k}(V(\down^q \up^p),
V(\up^r \down^s))
&@>>>&
\End_{\k}(V^{\otimes m}).
\end{CD}
\end{equation*}
The first horizontal maps are the bijections defined in an obvious way using the
symmetric braidings. Since the tensor functor 
$\Phi^\ell_\lambda(\delta)$ preserves the braidings,
the left hand square commutes.
The second horizontal maps are the bijections obtained using (\ref{dual hom}) and (\ref{dual hom 2}).
Since $\Phi^\ell_\lambda(\delta)$ preserves right duals, the right hand square commutes. 
Finally the right hand vertical map is injective according to the previous paragraph, hence the left hand vertical map is too.
\end{proof}

\subsection{\boldmath A filtered representation of $\COB^f(\delta_1,\dots,\delta_\ell)$}\label{repn AOB}  
In this subsection we are going to construct a filtered
functor
$\Psi_\lambda^f(\delta_1,\dots,\delta_\ell):
\COB^f(\delta_1,\dots,\delta_\ell) \rightarrow
\k\text{-fmod}$
which we will see is a deformation of the
graded functor $\Phi_\lambda^\ell(\delta)$ from (\ref{Cat}).
The definition of this functor 
will likely seem quite unmotivated; we will say more about its
origin in $\S$\ref{later}.
We begin by defining a 
functor
\begin{equation}\label{psifun}
\Psi_\lambda:\AOB\to\k\text{-fmod}.
\end{equation}
Note this is
{\em not} going to be a tensor functor. 
However, on forgetting the filtrations, it is going to agree
with the tensor functor $\Psi:\OB\to\k\text{-mod}$ from (\ref{Psi}) on objects and on all undotted diagrams. 
First, for $\ob a\in \wrd$, we 
set $\Psi_\lambda(\ob a) := V(\ob a)$ viewed as a filtered vector space
with $V(\ob a)_{\leq i} := \bigoplus_{j \leq i} V(\ob a)_j$.
This is the same underlying vector space as $\Psi(\ob a)$.
On morphisms,
it suffices to define $\Psi_\lambda$ on 
$\ob a^{(1)} \,c\,\ob a^{(2)}$,
$\ob a^{(1)} \,d\,\ob a^{(2)}$,
$\ob a^{(1)} \,s\,\ob a^{(2)}$,
$\ob a^{(1)} \,t\,\ob a^{(2)}$ and
$\ob a^{(1)} \,x\,\ob a^{(2)}$
for each $\ob a^{(1)}, \ob a^{(2)} \in \wrd$, since these morphisms generate
$\AOB$ as a $\k$-linear category.
For all but the last of these generating morphisms, we take the same linear 
map as given by
the functor $\Psi$ from $\S$\ref{OB action}.
It remains to define $\Psi_\lambda(\ob a^{(1)} \,x \, \ob a^{(2)})$.
Let $\ob a:= \ob a^{(1)} \up \ob a^{(2)}$ and assume that
$\ob a^{(1)}$ is of length $(p-1)$.
Then we set
\begin{equation}\label{x action} 
\Psi_\lambda(\ob a^{(1)} \,x\, \ob a^{(2)}) (v^{\ob a}_{\bd i}):=e_p(v^\ob a_{\bd i})
+m_{\col(i_p)}v^\ob a_{\bd i}+\sum_{q\not=p}(p,q)^\ob a_\lambda v^\ob a_{\bd i},
\end{equation} 
where $e_p$ is the map from (\ref{Lp}).

\begin{lemma}\label{Psi is a functor} 
The functor $\Psi_\lambda:\AOB \rightarrow \k\operatorname{-fmod}$ 
is well defined.
\end{lemma}

\begin{proof} 
We must check that $\Psi_\lambda$ respects the relations 
between the generating morphisms for $\AOB$ as a $\k$-linear category
which arise in the manner explained in \S\ref{g-r}.
All of the relations which do not involve dots follow from Corollary~\ref{OB functor}.  
The relations (\ref{AXdumb}) are easy to check using parts (iv) and (v) of Lemma \ref{t and (p,q)}.  The commuting relations involving $x$ and $s$ can be checked using parts (i)-(iii) of Lemma \ref{t and (p,q)}. 
The commuting relations involving $x$ and $c$ (resp.\ $x$ and $d$)
can be checked using
Lemma \ref{c and (p,q)} (resp. Lemma \ref{d and (p,q)}). It just 
remains to check the commuting relations involving two $x$'s.
Say $\ob a = \ob a^{(1)} \up \ob a^{(2)} = \ob b^{(1)} \up \ob b^{(2)}$
where $\ob a^{(1)}$ is of length $(p-1)$ and $\ob b^{(1)}$ is of length 
$(r-1)$ for $p \neq r$.
Then we need to show
that
\begin{equation}\label{xx} 
\Psi_\lambda(\ob a^{(1)} \,x\, \ob a^{(2)}) (\Psi_\lambda(\ob b^{(1)} \,x\, \ob b^{(2)}) (v_{\bd i}^{\ob a}))=\Psi_\lambda(\ob b^{(1)} \,x\, \ob b^{(2)})(\Psi_\lambda(\ob a^{(1)} \,x\, \ob a^{(2)})(v_{\bd i}^{\ob a}))
\end{equation}
 for all $\bd i$.
For this, we naively expand both sides using the definition (\ref{x action}); each side becomes a sum of nine terms.
Then we observe 
that
\begin{equation*}
(r,p)_\lambda^{\ob a}\circ e_p=-e_r\circ(p,r)_\lambda^{\ob a},
\end{equation*}
as is straightforward to verify from the definitions of the maps
(\ref{piano}) and (\ref{Lp}).
It follows that
the terms on the left and right hand sides of
the expansion of (\ref{xx}) that involve $e$'s are both equal to 
 $$
e_p(e_r(v^{\ob a}_{\bd i}))+m_{\col(i_p)}e_r(v^{\ob a}_{\bd i})+m_{\col(i_r)}e_p(v^{\ob a}_{\bd i})+\sum_{q\not=p,r}\left((p,q)^{\ob a}_\lambda e_r(v^{\ob a}_{\bd i})+(r,q)^{\ob a}_\lambda e_p(v^{\ob a}_{\bd i})\right).
$$
Of the remaining terms, if $\col(i_p)=\col(i_r)$ and we set $m=m_{\col(i_p)}=m_{\col(i_r)}$, then one can show the terms on the left and right hand sides of (\ref{xx}) involving $m$'s but no $e$'s are both equal to 
$$m^2v_{\bd i}^{\ob a}+m\sum_{q\not=p}(p,q)_\lambda^{\ob a}v^{\ob a}_{\bd i}+m\sum_{q\not=r}(r,q)_\lambda^{\ob a}v^{\ob a}_{\bd i}.
$$
On the other hand, if $\col(i_p)\not=\col(i_r)$ then $(r,p)_\lambda^{\ob a}v^{\ob a}_{\bd i}=-(p,r)_\lambda^{\ob a}v^{\ob a}_{\bd i}$.  Using this, it follows that
the terms on the left and right hand sides of (\ref{xx}) involving $m$'s but no $e$'s are both 
 $$m_{\col(i_p)}m_{\col(i_r)}v_{\bd i}^{\ob a}+m_{\col(i_r)}\sum_{q\not=p,r}(p,q)_\lambda^{\ob a}v^{\ob a}_{\bd i}+m_{\col(i_p)}\sum_{q\not=p,r}(r,q)_\lambda^{\ob a}v^{\ob a}_{\bd i}.$$ 
It just remains to observe that
the terms on both sides involving no $m$'s and no $e$'s are equal.
This follows from the identity
\begin{equation*} \sum_{\substack{q_1\not=p\\ q_2\not=r}}(p,q_1)^\ob a_\lambda\circ(r,q_2)^\ob a_\lambda=\sum_{\substack{q_1\not=p\\ q_2\not=r}}(r,q_2)^\ob a_\lambda\circ(p,q_1)^\ob a_\lambda,
\end{equation*}
which is a consequence of Lemma~\ref{mod (p,q) prop}.
\end{proof}

\begin{lemma}\label{step}
Given any $\ob a \in \wrd$ and $1 \leq k \leq \ell$,
let $g$ denote the image of
$(x\, \ob a - m_k (\up \ob a)) \circ (x\,\ob a - m_{k+1} (\up \ob a))
\circ \cdots \circ (x\, \ob a - m_\ell (\up \ob a))$ under the functor $\Psi_\lambda$.
Then for any $\bd i$, the vector
$g(v_\bd i^{\up \ob a})$ 
is contained in the subspace of $V(\up \ob a)$
spanned by all $v_\bd j^{\up \ob a}$ with $\col(j_1) < k$.
In particular if $k=1$ then $g=0$.
\end{lemma}

\begin{proof}
This follows from (\ref{x action}) using 
downward induction on $k=\ell,\dots,1$.
\end{proof}

\begin{lemma}\label{eta} 
Fix $1
\leq i, j\leq n$.
For each $k\geq 0$, let $\eta_{i,j}^{(k)}$ denote the coefficient of $v_i\otimes f_i$ in $\Psi_\lambda(x \down)^{k}(v_j\otimes f_j)$.  Then
$$
\eta_{i,j}^{(k)} = \left\{
\begin{array}{ll}
(m_{\col(i)})^{k} &\text{if $i=j$;}\\
\displaystyle
\sum_{r=1}^k
\sum_{\col(i)=p_0 < \cdots < p_r =\col(j)}
\!\!\!\!\!\!\!\!\!\!h_{k-r}(m_{p_0},\ldots,m_{p_{r}}) \lambda_{p_1} \cdots  \lambda_{p_{r-1}}
&\text{if $\col(i) < \col(j)$;}\\
0&\text{otherwise.}
\end{array}
\right.
$$
\end{lemma}

\begin{proof} As a special case of (\ref{x action}) we have that
\begin{equation}\label{x on up down}\Psi_\lambda(x\, \down)(v_j\otimes f_j) =
e(v_j) \otimes f_j+m_{\col(j)}v_j\otimes f_j+
\sum_{\col(i)<\col(j)}v_i\otimes f_i.
\end{equation}
It follows that $\eta_{i,j}^{(k)}=0$ unless $i=j$ or $\col(i)<\col(j)$, and that $\eta^{(k)}_{i,i}=m_{\col(i)}^{k}$.  
We now treat the case that $\col(i)<\col(j)$ by induction on $k$.  The base
case $k=0$ is clear.  For $k\geq 1$, we have by (\ref{x on up down}) that
\begin{align*}
\eta_{i,j}^{(k)}  &=m_{\col(j)}\eta_{i,j}^{(k-1)}+
\!\!\!\!\!\!\!\sum_{\col(h)<\col(j)}\!\!\!\!\!\eta^{(k-1)}_{i,h}=m_{\col(j)}\eta_{i,j}^{(k-1)}+m_{\col(i)}^{k-1}
+\!\!\!\!\!\!\!\!\!\sum_{\col(i)<\col(h)<\col(j)}\!\!\!\!\!\!\!\eta^{(k-1)}_{i,h}.
\end{align*}
Then using induction we get
\begin{align*}\eta_{i,j}^{(k)} &=
m_{\col(j)} 
\sum_{r=1}^{k-1}
\sum_{\col(i)= p_0 < \cdots < p_r =\col(j)}
h_{k-1-r}(m_{p_0},\dots,m_{p_r}) \lambda_{p_1}\cdots \lambda_{p_{r-1}}
+m_{\col(i)}^{k-1}\\
&\quad
+
\!\!\!\!\!\sum_{\col(i)<\col(h)<\col(j)}
\sum_{r=1}^{k-1}
\sum_{\col(i)= p_0 < \cdots < p_r=\col(h)}
h_{k-1-r}(m_{p_0},\dots,m_{p_r}) \lambda_{p_1}\cdots \lambda_{p_{r-1}}
\\
&=\sum_{r=1}^{k-1}\sum_{\col(i)= p_0 < \cdots < p_r=\col(j)}
h_{k-1-r}(m_{p_0},\dots,m_{p_r})m_{\col(j)} 
 \lambda_{p_1}\cdots \lambda_{p_{r-1}}
+m_{\col(i)}^{k-1}\\
&\quad
+
\sum_{\col(i)<p<\col(j)}
\sum_{r=1}^{k-1}
\sum_{\col(i)=p_0 < \cdots < p_r=p}
h_{k-1-r}(m_{p_0},\dots,m_{p_r}) \lambda_{p_1}\cdots \lambda_{p_{r-1}} \lambda_p
\\
&=\sum_{r=1}^{k-1}\sum_{\col(i)= p_0 < \cdots < p_r=\col(j)}
h_{k-1-r}(m_{p_0},\dots,m_{p_r})m_{\col(j)}  \lambda_{p_1}\cdots \lambda_{p_{r-1}}
+m_{\col(i)}^{k-1}\\
&\quad
+
\sum_{r=1}^{k-1}
\sum_{\col(i)= p_0 < \cdots < p_{r+1}=\col(j)}
h_{k-1-r}(m_{p_0},\dots,m_{p_r}) \lambda_{p_1}\cdots \lambda_{p_{r}}
\\
&=\sum_{r=1}^{k-1}\sum_{\col(i)= p_0 < \cdots < p_r=\col(j)}
h_{k-1-r}(m_{p_0},\dots,m_{p_r})m_{\col(j)}\lambda_{p_1}\cdots \lambda_{p_{r-1}}
\\
&\quad
+
\sum_{r=0}^{k-1}\sum_{\col(i)= p_0 < \cdots < p_{r+1} =\col(j)}
h_{k-1-r}(m_{p_0},\dots,m_{p_r}) \lambda_{p_1}\cdots \lambda_{p_{r}}
\\
&=
\sum_{r=1}^{k-1}
\sum_{\col(i)=p_0 < \cdots < p_r = \col(j)}
h_{k-1-r}(m_{p_0},\dots,m_{p_r})m_{\col(j)} \lambda_{p_1}\cdots \lambda_{p_{r-1}}
\\
&\quad
+
\sum_{r=1}^k
\sum_{\col(i) = p_0 < \cdots < p_r = \col(j)}
h_{k-r}(m_{p_0},\dots,m_{p_{r-1}}) \lambda_{p_1}\cdots \lambda_{p_{r-1}}
\\
&=\sum_{r=1}^k \sum_{\col(i) = p_0 < \cdots < p_r = \col(j)}
h_{k-r}(m_{p_0},\dots,m_{p_{r}}) \lambda_{p_1}\cdots \lambda_{p_{r-1}}.
\end{align*}
This is what we wanted.
\end{proof}

\begin{lemma}\label{deltaacts}
The image under $\Psi_\lambda$
of the clockwise bubble $\Delta_k$ with $(k-1)$ dots
is equal to the scalar $\delta_k$ from (\ref{deltadef}).
\end{lemma}

\begin{proof}
Since $\Delta_k=d'\circ (x \down)^{\circ(k-1)}\circ c$, 
it follows from Lemma~\ref{eta} that 
\begin{align}\notag
\Psi_\lambda(\Delta_k)&=
\sum_{i,j}\eta^{(k-1)}_{i,j}
=
\sum_{i}\eta^{(k-1)}_{i,i}+\sum_{\col(i)<\col(j)}\eta^{(k-1)}_{i,j}\\\notag
&
=\sum_{i}m_{\col(i)}^{k-1}
+
\sum_{r=1}^{k-1}
\sum_{\col(i)=p_0 < \cdots < p_r = \col(j)}
h_{k-1-r}(m_{p_0},\dots,m_{p_r})\lambda_{p_1} \cdots \lambda_{p_{r-1}}\\\notag
&
=\!\!\sum_{1\leq p_0\leq \ell}h_{k-1}(m_{p_0})\lambda_{p_0}+
\sum_{r=1}^{k-1}\sum_{1 \leq p_0< \cdots < p_r \leq \ell}
h_{k-1-r}(m_{p_0},\dots,m_{p_r})\lambda_{p_0} \cdots \lambda_{p_{r}}\\\notag
&
=
\sum_{r=0}^{k-1}\sum_{1 \leq p_0< \cdots < p_r \leq \ell}
h_{k-1-r}(m_{p_0},\dots,m_{p_r})\lambda_{p_0} \cdots \lambda_{p_{r}}\\\label{here}
&=
\sum_{r=1}^k\sum_{1 \leq p_1< \cdots < p_r\leq \ell}
h_{k-r}(m_{p_1},\dots,m_{p_r})\lambda_{p_1} \cdots \lambda_{p_{r}}.
\end{align}
Now we look at the formula for $\delta_k$ from (\ref{deltadef}).
Viewing it as a polynomial in indeterminates $\lambda_1,\dots,\lambda_\ell$,
it is clear that it is a linear combination of monomials of the form
$\lambda_{p_1} \cdots \lambda_{p_r}$
for $1 \leq p_1 < \cdots < p_r \leq \ell$ and $r \geq 0$.
To compute the coefficient of $\lambda_{p_1} \cdots \lambda_{p_r}$,
let $q_1 < \cdots < q_{\ell-r}$
be defined so that $\{p_1,\dots,p_r,q_1,\dots,q_{\ell-r}\} = \{1,\dots,\ell\}$.
Then it is easy to see from (\ref{deltadef}) that the
$\lambda_{p_1} \cdots \lambda_{p_r}$-coefficient of $\delta_k$
is equal to
$$
\sum_{i+j=k-r}
(-1)^{j} h_i(m_1,...,m_\ell) e_{j} (m_{q_1},\dots,m_{q_{\ell-r}}).
$$
But now, by a standard identity (e.g. see \cite[(2.4)]{Bspringer}), this is zero in case $r=0$,
while for $r > 0$ it simplifies to $h_{k-r}(m_{p_1},\dots,m_{p_r})$.
This is the same as the coefficient in (\ref{here}), so the lemma is proved.
\end{proof}

\begin{theorem}\label{repthm}
The functor $\Psi_\lambda:\AOB \rightarrow \k\operatorname{-fmod}$
factors through the quotient $\COB^f(\delta_1,\dots,\delta_\ell)$
to induce a functor
\begin{equation}
\Psi_\lambda^f(\delta_1,\dots,\delta_\ell):
\COB^f(\delta_1,\dots,\delta_\ell)
\rightarrow \k\operatorname{-fmod}.
\end{equation}
Moreover, this functor
is filtered, and the associated graded functor
fits into the following commuting diagram:
$$
\begin{CD}
\TOB^\ell(\delta)&@>\phantom{xxxl}\Phi^\ell_\lambda(\delta)\phantom{xxxl}>>&\k\operatorname{-gmod}\\
@V\Theta^f(\delta_1,\dots,\delta_\ell)VV&&@AAG A\\
\grOB^f(\delta_1,\dots,\delta_\ell)&@>>\gr \Psi^f_\lambda(\delta_1,\dots,\delta_\ell)>&
\gr (\k\operatorname{-fmod}),
\end{CD}
$$
where $G$ is the canonical functor from (\ref{Gfunc}).
\end{theorem}

\begin{proof}
The last assertion of Lemma~\ref{step} 
implies that the functor $\Psi_\lambda$ annihilates the right tensor ideal of
$\AOB$ generated by $f(x)$.
Hence it factors through the quotient
$\COB^f$ of $\AOB$ to induce a functor
\begin{equation}
\Psi_\lambda^f:\COB^f \rightarrow \k\operatorname{-fmod}.
\end{equation}
By Lemma~\ref{deltaacts} this functor maps $\Delta_k$ to $\delta_k$,
so it factors further through the specialization
$\COB^f(\delta_1,\dots,\delta_\ell)$ as desired.

To see that $\Psi^f_\lambda(\delta_1,\dots,\delta_\ell)$ is a filtered
functor, the action of $\OB$ 
is clearly homogeneous of degree zero. 
Also the three terms on the right hand side of (\ref{x action})
are graded maps of degrees one, zero, and zero, respectively.
Hence the action of $x$ is in filtered degree one.
Finally to see that the given diagram commutes, it suffices to check it on each of the generating morphisms of $\TOB^\ell(\delta)$ in turn.
This is clear for the generators of degree zero.
It just remains to observe that the degree one term on the right hand side of
(\ref{x action})
is exactly the map 
from (\ref{Lp}).
\end{proof}

\begin{corollary}\label{mainpoint}
Let $\ob a, \ob b \in \wrd$ be words whose average length is $\leq \min(\lambda)$.
Then the linear map $\Hom_{\TOB^\ell(\delta)}(\ob a, \ob b)
\rightarrow \Hom_{\grOB^f(\delta_1,\dots,\delta_\ell)}(\ob a, \ob b)$ 
induced by $\Theta^f(\delta_1,\dots,\delta_\ell)$
is an isomorphism.
\end{corollary}

\begin{proof}
We already observed that this map is surjective at the end of the previous section.
It is injective thanks to 
the commutative diagram from Theorem~\ref{repthm} together with Lemma~\ref{vustinj}.
\end{proof}

\begin{corollary}\label{mainpoint2}
Let $\ob a, \ob b \in \wrd$ be words whose average length is $\leq \min(\lambda)$.
Then $\Hom_{\COB^f(\delta_1,\dots,\delta_\ell)}(\ob a, \ob b)$ is a free $\k$-module
with basis arising from the equivalence classes of normally ordered dotted
oriented Brauer diagrams of type $\ob a\rightarrow \ob b$ such that there are at most $(\ell-1)$ dots on each strand.
\end{corollary}

\begin{proof}
It suffices to show that the graded morphisms defined by these diagrams
give a basis for the associated graded Hom space
$\Hom_{\operatorname{gr} \COB^f(\delta_1,\dots,\delta_\ell)}(\ob a, \ob b)$.
This follows from Corollary~\ref{mainpoint}, since the corresponding morphisms
clearly give a basis for $\Hom_{\TOB^\ell(\delta)}(\ob a, \ob b)$
by the definition of the category $\TOB^\ell(\delta)$.
\end{proof}

\subsection{\boldmath Remarks about the connection to finite $W$-algebras}\label{later}
The material in this subsection is not needed elsewhere in the article.
The goal is give a brief sketch of the origin of the functor
$\Psi_\lambda:\AOB \rightarrow \k\operatorname{-fmod}$.
We assume that $\k$ is an algebraically closed field of characteristic zero.

Let us start with some more discussion of the graded picture.
Let $\mathfrak g := \mathfrak{gl}_n(\k)$ with natural module $V$
and dual natural module $V^*$.
Define a $\Z$-grading $$
\mathfrak{g} = \bigoplus_{d \in \Z} \mathfrak{g}_d
$$
 by declaring 
that the $ij$-matrix unit $e_{i,j}$ is of degree $\col(j)-\col(i)$.
The gradings on $V$ and $V^*$ from $\S$\ref{thenilpotent}
make them into 
graded $\mathfrak{g}$-modules,
as are all of the tensor products $V(\ob a)$.
Let $e \in \mathfrak g_1$ be the nilpotent
matrix from $\S$\ref{thenilpotent}.
Its centralizer
$\mathfrak{g}^e$ is a graded subalgebra of $\mathfrak{g}$.
Let $\mathfrak{g}^e\operatorname{-gmod}$ denote 
the graded monoidal category consisting of all
graded $\mathfrak{g}^e$-modules, with morphisms defined like we did for the category
$\k\operatorname{-gmod}$ in $\S$\ref{filters}.
The graded functor $\Phi_\lambda:\GOB \rightarrow \k\operatorname{-gmod}$ 
from $\S$\ref{thenilpotent}
should really be viewed as a graded
functor $\overline{\Phi}_\lambda:\GOB \rightarrow \mathfrak{g}^e\operatorname{-gmod}$.
Then, letting $F:\mathfrak{g}^e\operatorname{-gmod}\rightarrow \k\operatorname{-gmod}$
be the obvious forgetful functor, the following diagram commutes:
\begin{equation}\label{swing1}
\xymatrix{\GOB\ar[r]^{\Phi_\lambda}\ar[d]_{\overline{\Phi}_\lambda} & \k\operatorname{-gmod}\cr
\mathfrak{g}^e\operatorname{-gmod}\ar[ur]_{F}}.
\end{equation} 
The next theorem is a reformulation 
of a result of Vust
proved in \cite[$\S$6]{KP}.

\begin{theorem}\label{full}
The functor $\overline{\Phi}_\lambda:\GOB \rightarrow
\mathfrak{g}^e\operatorname{-gmod}$ is full.
\end{theorem}

\begin{proof}
We must show 
for all $\ob a, \ob b \in \wrd$
that it maps
$\Hom_{\GOB}(\ob a, \ob b)$ surjectively onto
$\Hom_{\mathfrak{g}^e}(V(\ob a), V(\ob b))$.
In the case that $\ob a = \ob b = \up^m$ for some $m \geq 0$,
this follows from Vust's theorem as formulated e.g. in
\cite[Theorem 2.4]{BK08}.
For the general case,
suppose that $\ob a$ involves $p$ $\up$'s and $q$ $\down$'s,
while $\ob b$ involves $r$ $\up$'s and $s$ $\down$'s.
There is nothing to prove unless
$m :=p+s=r+q$, as both Hom spaces are zero
if that is not the case.
Then we get done by the special case 
just treated using the following commuting diagram with bijective rows:
\begin{equation*}
\begin{CD}
\!\!\Hom_{\GOB}(\ob a, \ob b)&@>>>&
\Hom_{\GOB}(\down^q \up^p, \up^r \down^s)
&@>>>&
\End_{\GOB}(\up^{m})\\
@V\overline{\Phi}_\lambda VV&
&@V\overline{\Phi}_\lambda VV&
&@VV\overline{\Phi}_\lambda V\\
\!\!\Hom_{\mathfrak{g}^e}(V(\ob a), V(\ob b))&@>>>&
\Hom_{\mathfrak{g}^e}(V(\down^q \up^p),
V(\up^r \down^s))
&@>>>&
\End_{\mathfrak{g}^e}(V^{\otimes m}).
\end{CD}
\end{equation*}
This is defined in exactly the same way as the similar looking diagram
from the proof of Lemma~\ref{vustinj}.
\end{proof}

The universal enveloping algebra 
$U(\mathfrak{g}^e)$ admits a certain filtered deformation $U(\g,e)$, namely,
the {\em finite $W$-algebra} associated
to the nilpotent matrix $e$.
This algebra is naturally filtered in such a way that the associated graded algebra
$\gr U(\g,e)$ is identified with $U(\mathfrak{g}^e)$,
as is explained in detail in \cite[$\S$3.1]{BK08}.
Let $U(\g,e)\operatorname{-fmod}$ be the category of 
filtered $U(\g,e)$-modules, with morphisms defined like we did for
the category
$\k\operatorname{-fmod}$ from $\S$\ref{filters}.
Note for any $M \in U(\g,e)\operatorname{-fmod}$ that
the associated graded module $\gr M$ is naturally a graded $\mathfrak{g}^e$-module.
Hence there is a canonical faithful functor
$G:\gr (U(\g,e)\operatorname{-fmod}) \rightarrow \mathfrak{g}^e\operatorname{-gmod}$
defined in the same way as (\ref{Gfunc}).
Writing $F:U(\g,e)\operatorname{-fmod} \rightarrow \k\operatorname{-fmod}$
for the forgetful functor, the following diagram obviously commutes:
\begin{equation}\label{prettystupid}
\begin{CD}
\gr(U(\g,e)\operatorname{-fmod})&@>\gr F >>&\gr(\k\operatorname{-fmod})\\
@VG VV&&@VVG V\\
\mathfrak{g}^e\operatorname{-gmod}&@>>F>&\k\operatorname{-gmod}.
\end{CD}
\end{equation}

Now the point is that there is a filtered functor $\overline{\Psi}_\lambda:\AOB
\rightarrow U(\g,e)\operatorname{-fmod}$ making the following diagram commute:
\begin{equation}\label{swing2}
\xymatrix{\AOB\ar[r]^{\Psi_\lambda}\ar[d]_{\overline{\Psi}_\lambda} & \k\operatorname{-fmod}\cr
U(\g,e)\operatorname{-fmod}\ar[ur]_{F}}.
\end{equation} 
In other words, for each $\ob a \in \wrd$,
there is a natural filtered action of $U(\g,e)$ on 
$V(\ob a)$ (which is $\Psi_\lambda(\ob a)$), and the functor $\Psi_\lambda$ takes morphisms
in $\AOB$ to $U(\g,e)$-module homomorphisms.
Without going into details, 
this arises 
by using an analog of \cite[(3.9)]{BK08} to identify $V(\ob a)$ with a
naturally occurring $U(\g,e)$-module.
This $U(\g,e)$-module is the image
under the so-called
Skryabin equivalence of a certain infinite dimensional
generalized Whittaker module for $\g$ 
(depending on the scalars $m_1,\dots,m_\ell$)
tensored 
with the finite dimensional $\g$-module $V(\ob a)$.
Then the actions of $c, d, s$ and $x$
arise naturally by
pushing the endomorphisms 
defined by (\ref{F}) through Skryabin's equivalence. 
We used this point of view
to discover the functor $\Psi_\lambda$ in the first place;
the formula (\ref{x action}) 
was computed
by mimicking the proof of \cite[Lemma 3.3]{BK08}.

\begin{remark}
Using Theorems~\ref{12prime} and \ref{repthm} 
plus the definition of the functor $\overline{\Psi}_\lambda$,
one can show that (\ref{swing2}) is a filtered deformation of 
(\ref{swing1}).
Formally, this means that there is a commuting triangular prism of functors with top face being obtained from (\ref{swing2}) by applying $\gr$,
bottom face being (\ref{swing1}), one of the side faces being (\ref{prettystupid}),
and the remaining vertical edge being the isomorphism
$\Theta$ from Theorem~\ref{12prime}.
Combined with Theorem~\ref{full},
one can deduce from this
that the functor $\overline{\Psi}_\lambda$ is full too.
This assertion generalizes the second equality 
from \cite[Theorem 3.7]{BK08}.
\end{remark}

\section{Main results}\label{b-t}  

\subsection{\boldmath Proof of Theorem~\ref{thm3} assuming $\k$ is an algebraically closed 
field of characteristic zero}
Let $\k$ be an algebraically closed field of characteristic zero
and $f(u) \in \k[u]$ be an arbitrary monic polynomial of degree $\ell$.
We can factor
$f(u) = (u-m_1)\cdots (u-m_\ell)$ for $m_1,\dots,m_\ell \in \k$.
For $\ob a, \ob b \in \wrd$,
let $D^\ell(\ob a, \ob b)$ be a set of representatives for the
equivalence classes of normally ordered dotted oriented Brauer diagrams
of type $\ob a \rightarrow \ob b$ 
such that there are at most $(\ell-1)$ dots on each strand.
View $\Hom_{\COB^f}(\ob a, \ob b)$
as a module over the 
polynomial algebra 
$\k[t_1,\dots,t_\ell]$
so that $t_i$ acts on a morphism
by tensoring on the left with the clockwise dotted bubble $\Delta_i$ 
with $(i-1)$ dots.
We have observed already that 
the morphisms arising from the diagrams in $D^\ell(\ob a, \ob b)$ 
span
$\Hom_{\COB^f}(\ob a, \ob b)$ as a $\k[t_1,\dots,t_\ell]$-module.
To complete the proof of Theorem~\ref{thm3}
we need to show that these morphisms 
are linearly independent over $\k[t_1,\dots,t_\ell]$.
Take a $\k[t_1,\dots,t_\ell]$-linear relation
$$
\sum_{g \in D^\ell(\ob a, \ob b)}
p_g(t_1,\dots,t_\ell) g = 0
$$
in $\Hom_{\COB^f}(\ob a, \ob b)$.
We need to show that all of the polynomials
$p_g(t_1,\dots,t_\ell)$ are identically zero.

To see this, suppose that $\lambda = (\lambda_1,\dots,\lambda_\ell)$
is any unimodular sequence 
such that all of its parts are greater than or equal to
the average length of the words $\ob a$ and $\ob b$.
Define $\delta=\delta_1,\dots,\delta_\ell \in \k$ according to (\ref{deltadef}).
Specializing each $\Delta_k$ at $\delta_k$, we deduce that
$$
\sum_{g \in D^\ell(\ob a, \ob b)}
p_g(\delta_1,\dots,\delta_\ell) g = 0
$$
in 
$\Hom_{\COB^f(\delta_1,\dots,\delta_\ell)}(\ob a, \ob b)$.
Hence by Corollary~\ref{mainpoint2}
we have that 
$p_g(\delta_1,\dots,\delta_\ell) = 0$ for all $g$.
It just remains to apply the following lemma.

\begin{lemma}
As $(\lambda_1,\dots,\lambda_\ell)$ varies over all unimodular sequences of length $\ell$
having all parts greater than or equal to
the average length $r$ of the words $\ob a$ and $\ob b$,
the set of points $(\delta_1,\dots,\delta_\ell)$ 
defined by the equation (\ref{deltadef})
is Zariski dense in $\k^\ell$.
\end{lemma}

\begin{proof}
The set of unimodular sequences of length $\ell$ 
having all parts greater than or equal to $r$
is Zariski dense in $\k^\ell$.
Hence it suffices to show that the morphism
$$
\k^\ell \rightarrow \k^\ell,\qquad
(\lambda_1,\dots,\lambda_\ell) \mapsto (\delta_1,\dots,\delta_\ell)
$$
defined by (\ref{deltadef}) is dominant.
To prove this, we just need to check that the determinant of the Jacobian matrix
$J = \big(\frac{\partial \delta_i}{\partial \lambda_j}\big)_{1 \leq i,j \leq \ell}$
is non-zero.
Each $\delta_i$ is a polynomial of degree $i$ in $\k[\lambda_1,\dots,\lambda_\ell]$,
and the homogeneous component of $\delta_i$ of this top degree
is equal to $e_i(\lambda_1,\dots,\lambda_\ell)$.
Hence each $\frac{\partial \delta_i}{\partial \lambda_j}$
is a polynomial of degree $(i-1)$ with top homogeneous component
$e_{i-1}(\lambda_1,\dots,\widehat{\lambda_j},\dots,\lambda_\ell)$.
We deduce that the top homogeneous component of $\det J$ is
equal to the determinant of the matrix 
$\big(e_{i-1}(\lambda_1,\dots,\widehat{\lambda_j},\dots,\lambda_\ell)\big)_{1\leq i,j \leq \ell}$.
This is a variation on Vandermonde, equal to $\prod_{1 \leq i<j \leq \ell}(\lambda_i-\lambda_j)$,
which is non-zero.
\end{proof}

\subsection{\boldmath Proof of Theorem~\ref{thm3} assuming $\k$ is an integral domain of 
characteristic zero}
Let $\k$ be a domain of characteristic zero and $\K$ be the algebraic closure
of its field of fractions.
Let $f(u) \in \k[u] \subseteq \K[u]$ be a monic polynomial of degree $\ell$.
To prove Theorem~\ref{thm3} in this case, we just need to compare
the category $\OB_\k^f$ defined over the ground ring $\k$
to the category $\OB_\K^f$ defined over the algebraically closed field $\K$.
There is an obvious $\k$-linear functor
$\COB^f_\k \rightarrow \COB^f_\K$ mapping generators to generators.
It sends the morphisms in $\Hom_{\COB^f_\k}(\ob a, \ob b)$ 
defined by the equivalence classes of diagrams
from the statement of Theorem~\ref{thm3} to the analogous
morphisms in $\Hom_{\COB^f_\K}(\ob a, \ob b)$.
The latter are known already to be $\K$-linearly independent
by the previous subsection.
Hence the given morphisms in $\Hom_{\COB^f_\k}(\ob a, \ob b)$
are $\k$-linearly independent, and they obviously span. This completes the proof.

\subsection{\boldmath Proof of Theorem~\ref{thm3} in general}
Finally we let $\k$ be an integral domain of positive characteristic
and $f(u) \in \k[u]$ be a monic polynomial of degree $\ell$.
By some general nonsense (e.g. see \cite[$\S$3.1]{Kap}), there exists 
an integral domain $\mathcal O$ of characteristic zero and a maximal ideal $\mathfrak{m} \lhd \mathcal O$ so that 
$\k$ embeds into the field $\K := \mathcal O / \mathfrak m$. 
Let $\hat f(u) \in \mathcal O[u]$ be a monic polynomial
whose image in $\K[u]$ is equal to the image of $f(u)$.
By the previous subsection, we have proved Theorem~\ref{thm3}
already for the category $\COB^{\hat f}_{\mathcal O}$.
Base change gives us a category
$\K \otimes_{\mathcal O} \COB^{\hat f}_{\mathcal O}$
in which the images of the morphisms from the statement of Theorem~\ref{thm3}
are $\K$-linearly independent.
Considering the obvious functor
$\COB^f_{\k} \rightarrow 
\K \otimes_{\mathcal O} \COB^{\hat f}_{\mathcal O}$, we deduce that the corresponding morphisms in $\COB^f_{\k}$
are $\k$-linearly independent,
as required to complete the proof of Theorem~\ref{thm3} in general.
As discussed in $\S$\ref{GOB}, Theorem~\ref{15prime} follows too,
as does the fact that the functor (\ref{Actual}) is an isomorphism in all cases.

\subsection{\boldmath Proof of Theorem~\ref{thm2}}
We just need to show that the morphisms from the statement of Theorem~\ref{thm2} are 
linearly independent.
Suppose for a contradiction 
that we are given some non-trivial linear relation between some of these
morphisms. Choose $\ell$ so that each strand of each of the diagrams 
involved in this linear relation has less than $\ell$ dots on it.
Then pick any monic $f(u) \in \k[u]$ of degree $\ell$ and
apply the quotient functor
$\AOB \rightarrow \COB^f$ to the given relation. 
The result is a non-trivial linear relation between
morphisms in $\COB^f$ which are already known by Theorem~\ref{thm3}
to be linearly independent.
This contradiction completes the proof of Theorem~\ref{thm2}.
Theorem~\ref{12prime} also follows.

\subsection{Identification with the algebras of Rui and Su}
We have now proved all of the main results stated in the introduction.
To conclude the article, we explain in more detail how to see that
the affine and cyclotomic walled Brauer algebras 
$AB_{r,s}(\delta_1,\delta_2,\dots)$ 
and $B_{r,s}^f(\delta_1,\dots,\delta_\ell)$
from (\ref{AB})--(\ref{Brs})
are isomorphic to the algebras with the same names
defined in \cite{RS, RS2}.

Consider first the affine walled Brauer algebra
$AB_{r,s}(\delta_1,\delta_2,\dots)$.
Let $B_{r,s}^{\operatorname{aff}}$
be the algebra from \cite[Definition 2.7]{RS}
over the ground ring $\k$
taking the parameters $\omega_0$ 
and $\omega_1$ there to be $\delta_1$ and $-\delta_2$, respectively.
This means that
$B_{r,s}^{\operatorname{aff}}$
is defined by
generators 
$e_1,x_1,\bar x_1,
s_i\:(i=1,\dots,r-1),
\bar s_j\:(j=1,\dots,s-1), 
\omega_k\:(k \geq 2)$ and $\bar \omega_k\:(k \geq 0)$,
subject to the relations that all
$\omega_k$ and $\bar \omega_k$ are central plus twenty-six more.
It is an exercise in checking relations to see that there is a 
well-defined homomorphism
$B^{\operatorname{aff}}_{r,s} 
\rightarrow AB_{r,s}(\delta_1,\delta_2,\dots)$
defined by
$$\begin{array}{ll}
\includegraphics{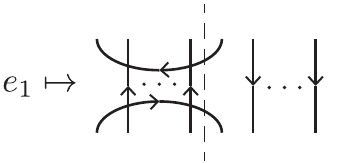}\\
\includegraphics{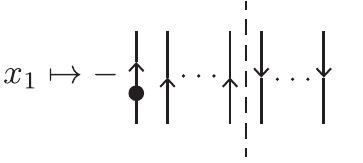} & \includegraphics{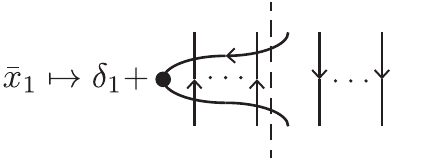}\\
\includegraphics{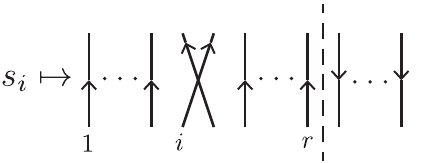} & \includegraphics{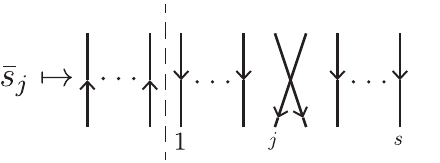}\\
\omega_k \mapsto (-1)^k \delta_{k+1} & \bar \omega_k \mapsto\sum_{l=0}^k\binom{k}{l} \delta_1^l \delta_{k+1-l}'\\[4pt]
& \text{where $\delta_j'$ is  defined from the}\\
& \text{identity (\ref{delta id}).} 
\end{array}$$
This homomorphism factors through the quotient
$\widehat{B}_{r,s}$ of
${B}^{\operatorname{aff}}_{r,s}$ 
by the additional relations 
$\omega_k = (-1)^k \delta_{k+1}$ for each $k$,
which is precisely the specialized algebra appearing in 
\cite[Theorem 4.15]{RS}.
Then one checks easily using our basis from
Theorem~\ref{thm2}
that the spanning set for $\widehat{B}_{r,s}$
defined in \cite[Theorem 4.15]{RS} maps to a basis for $AB_{r,s}(\delta_1,\delta_2,\dots)$. Hence $\widehat{B}_{r,s}
\cong AB_{r,s}(\delta_1,\delta_2,\dots)$. (This also
gives another proof of the linear independence in \cite[Theorem 4.15]{RS}.)

Finally we discuss the cyclotomic walled Brauer algebra.
Let 
$f(u) = \sum_{i=0}^\ell a_i u^{\ell-i}$ be
monic of degree $\ell$ as usual
and $\omega_0,\dots,\omega_{\ell-1} \in \k$
be some given scalars.
For $k \geq \ell$ define $\omega_k$ recursively from the equation
$\omega_k = -(a_1 \omega_{k-1} + \cdots + a_{\ell} \omega_{k-\ell}).$
Also let $\tilde f(u) := (-1)^{\ell} f(-u)$ and $\delta_{k+1}  := (-1)^k
\omega_k$ for each $k \geq 0$.
In \cite[Definition 2.1]{RS2},
Rui and Su define
their cyclotomic walled Brauer algebra
$B_{\ell,r,s}$ to be the quotient of the algebra $\widehat{B}_{r,s}$
from the previous paragraph by the additional relations that $f(x_1)
= g(\bar x_1) = 0$,
where $g(u)$ is another  monic polynomial defined explicitly via the
identity \cite[(2.6)]{RS2}.
Using Remark~\ref{rem3}, one can check that the composition of the isomorphism 
$\widehat{B}_{r,s} \stackrel{\sim}{\rightarrow}
AB_{r,s}(\delta_1,\delta_2,\dots)$ from the previous paragraph
with the natural 
quotient map
$AB_{r,s}(\delta_1,\delta_2,\dots) \twoheadrightarrow
B_{r,s}^{\tilde f}(\delta_1,\dots,\delta_\ell)$
sends both $f(x_1)$ and $g(\bar x_1)$ to zero.
Hence it
factors through $B_{\ell,r,s}$ to induce 
a surjection $B_{\ell,r,s} \twoheadrightarrow
B_{r,s}^{\tilde f}(\delta_1,\dots,\delta_\ell)$.
On comparing the spanning set for $B_{\ell,r,s}$
derived in the proof of \cite[Theorem 2.12]{RS2}
with our basis for $B_{r,s}^{\tilde f}(\delta_1,\dots,\delta_\ell)$
arising from Theorem~\ref{thm3}, it follows that this surjection is actually an isomorphism. Hence our cyclotomic walled Brauer algebra is the same as the one in
\cite{RS2}.

\renewcommand{\bibname}{\textsc{references}} 
\bibliographystyle{alphanum}	
\bibliography{references}	

\end{document}